\documentclass[a4paper]{article}
\title{\vspace{-2cm}The linear Tur\'an number of the 3-graph $P_5$}
\author{Chaoliang Tang, Hehui Wu, Junchi Zhang}
\date{\today}
\usepackage[utf8]{inputenc}
\usepackage{amsthm}
\usepackage{amsmath} 
\usepackage{float}
\usepackage{amssymb}
\usepackage{csquotes}
\usepackage{amsfonts,amssymb}
\usepackage{booktabs}
\usepackage{diagbox}
\usepackage{float}
\usepackage{tikz-cd}
\usepackage{geometry}
\usepackage{hyperref}
\usepackage{cleveref}
\usepackage{verbatim}
\usepackage{caption}
\usepackage{subcaption}
\newtheorem{thm}{Theorem}[section]

\newtheorem{lem}[thm]{Lemma}
\newtheorem{coro}[thm]{Corollary}

\newtheorem{prop}[thm]{Proposition}
\newtheorem{conj}[thm]{Conjecture}
\newtheorem{defi}[thm]{Definition}

\newtheorem{exmp}[thm]{Example}

\newtheorem{clm}[thm]{Claim}
\usetikzlibrary{cd}
\usepackage {indentfirst}

\newcommand{\ti}{\textit}
\newcommand{\emphd}[1]{{\fontseries{b}\selectfont\textsf{#1}}}

\usepackage{caption}
\usetikzlibrary{cd}
\usepackage {indentfirst}

\begin{document}

\maketitle

\begin{abstract}
We prove that for any linear 3-graph on $n$ vertices without a path of length 5, the number of edges is at most $\frac{15}{11}n$, and the equality holds if and only if the graph is the disjoint union of $G_0$, a graph with 11 vertices and 15 edges. Thus, $ex_L(n,P_5)\leq \frac{15}{11}n$, and the equality holds if and only if $11|n$.
\end{abstract}


\section{Introduction}
A \emphd{$k$-uniform hypergraph}, or a \emphd{$k$-graph}, is a pair $(V,E)$ such that $V$ is the set of vertices, and $E$ is a family of $k$-vertex subsets of $V$. A \emphd{linear 3-graph} $G=(V,E)$ is a 3-uniform hypergraph such that any two vertices are contained in at most one edge. Given a linear {3-graph} $F$, the \emphd{linear Tur\'an number} $ex_L(n,F)$ is the maximum number of edges in an $F$-free linear 3-graph on $n$ vertices.

Determining the Tur\'an number is a fundamental problem in extremal graph theory. A celebrated result is the famous (6,3)-theorem proved by Ruzsa and Szemer\'edi~\cite{MR519318}  that $ex_L(n,C_{3})=o(n^2)$, where the cycle $C_3$ is the unique linear 3-graph with 3 edges and 6 vertices. Determining the linear Tur\'an number of graphs seems surprisingly hard, even for {acyclic cases}.
Recently Gyárfás {et al.} initiated the study of linear Tur\'an number of some trees in~\cite{MR4315540}, where Gyárfás et al. proved that $ex_L(n,P_k)\leq 1.5kn$ and determined the linear Tur\'an number of $P_3, P_4, B_4$ and matchings. In~\cite{MR4477845}, Tang et al. proved that the linear Tur\'an number of the Crown graph $E_4$ is $1.5n$ and completed the determination of linear Tur\'an number for acyclic 3-graphs with at most 4 edges. The graphs mentioned are demonstrated as below.

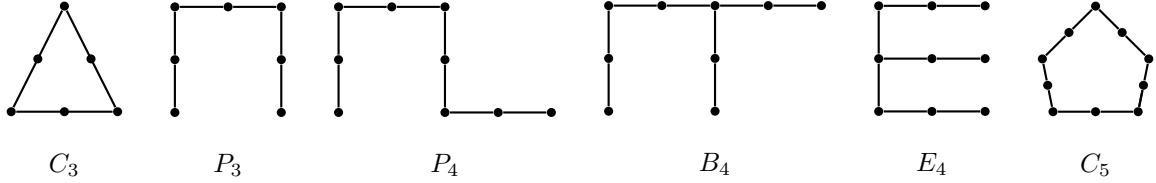
\begin{figure}[h] 
	\centering
    \begin{tikzpicture}[scale=0.7, every node/.style={circle, fill, inner sep=1.2pt}]
	
	\node[fill=black] (v1)at(-1,0) {};
	\node[fill=black] (v2) at (0, 0) {};
	\node[fill=black] (v3) at (1, 0) {};
	\node[fill=black] (v4) at (0, 2) {};
	\node[fill=black] (v5) at (-0.5, 1) {};
	\node[fill=black] (v6) at (0.5, 1) {};

	\draw[black, thick] (v1) -- (v3);
    \draw[black, thick] (v4) -- (v3);
    \draw[black, thick] (v1) -- (v4);
	\node[draw=none,fill=none] at (0, -1) {$C_3$};
\end{tikzpicture}
	\hspace{0.4cm}
	\begin{tikzpicture}[scale=0.7, every node/.style={circle, fill, inner sep=1.2pt}]
	
	\node[fill=black] (v1)at(-1,-1) {};
	\node[fill=black] (v2) at (-1, 0) {};
	\node[fill=black] (v3) at (-1, 1) {};
	\node[fill=black] (v4) at (0, 1) {};
	\node[fill=black] (v5) at (1, 1) {};
	\node[fill=black] (v6) at (1, 0) {};
        \node[fill=black] (v7) at (1, -1) {};
	\draw[black, thick] (v1) -- (v2);
        \draw[black, thick] (v2) -- (v3);
        \draw[black, thick] (v3) -- (v4);
        \draw[black, thick] (v4) -- (v5);
        \draw[black, thick] (v5) -- (v6);
        \draw[black, thick] (v6) -- (v7);
	\node[draw=none,fill=none] at (0, -2) {$P_3$};
\end{tikzpicture}
	\hspace{0.4cm}
		\begin{tikzpicture}[scale=0.7, every node/.style={circle, fill, inner sep=1.2pt}]
	
	\node[fill=black] (v1)at(-1,-1) {};
	\node[fill=black] (v2) at (-1, 0) {};
	\node[fill=black] (v3) at (-1, 1) {};
	\node[fill=black] (v4) at (0, 1) {};
	\node[fill=black] (v5) at (1, 1) {};
	\node[fill=black] (v6) at (1, 0) {};
        \node[fill=black] (v7) at (1, -1) {};
        \node[fill=black] (v8) at (2, -1) {};
        \node[fill=black] (v9) at (3, -1) {};
	\draw[black, thick] (v1) -- (v2);
        \draw[black, thick] (v2) -- (v3);
        \draw[black, thick] (v3) -- (v4);
        \draw[black, thick] (v4) -- (v5);
        \draw[black, thick] (v5) -- (v6);
        \draw[black, thick] (v6) -- (v7);
        \draw[black, thick] (v7) -- (v8);
        \draw[black, thick] (v8) -- (v9);
	\node[draw=none,fill=none] at (1, -2) {$P_4$};
\end{tikzpicture}
\hspace{0.4cm}
		\begin{tikzpicture}[scale=0.7, every node/.style={circle, fill, inner sep=1.2pt}]
	
	\node[fill=black] (v1)at(-1,-1) {};
	\node[fill=black] (v2) at (-1, 0) {};
	\node[fill=black] (v3) at (-1, 1) {};
	\node[fill=black] (v4) at (0, 1) {};
	\node[fill=black] (v5) at (1, 1) {};
	\node[fill=black] (v6) at (1, 0) {};
        \node[fill=black] (v7) at (1, -1) {};
        \node[fill=black] (v8) at (2, 1) {};
        \node[fill=black] (v9) at (3, 1) {};
	\draw[black, thick] (v1) -- (v2);
        \draw[black, thick] (v2) -- (v3);
        \draw[black, thick] (v3) -- (v4);
        \draw[black, thick] (v4) -- (v5);
        \draw[black, thick] (v5) -- (v6);
        \draw[black, thick] (v6) -- (v7);
        \draw[black, thick] (v5) -- (v8);
        \draw[black, thick] (v8) -- (v9);
	\node[draw=none,fill=none] at (1, -2) {$B_4$};
\end{tikzpicture}
\hspace{0.4cm}
		\begin{tikzpicture}[scale=0.7, every node/.style={circle, fill, inner sep=1.2pt}]
	
	\node[fill=black] (v1)at(-1,-1) {};
	\node[fill=black] (v2) at (-1, 0) {};
	\node[fill=black] (v3) at (-1, 1) {};
	\node[fill=black] (v4) at (0, 1) {};
	\node[fill=black] (v5) at (1, 1) {};
	\node[fill=black] (v6) at (1, 0) {};
        \node[fill=black] (v7) at (1, -1) {};
        \node[fill=black] (v8) at (0, 0) {};
        \node[fill=black] (v9) at (0, -1) {};
	\draw[black, thick] (v1) -- (v2);
        \draw[black, thick] (v2) -- (v3);
        \draw[black, thick] (v3) -- (v4);
        \draw[black, thick] (v4) -- (v5);
        \draw[black, thick] (v2) -- (v8);
        \draw[black, thick] (v8) -- (v6);
        \draw[black, thick] (v1) -- (v9);
        \draw[black, thick] (v7) -- (v9);
	\node[draw=none,fill=none] at (0, -2) {$E_4$};
\end{tikzpicture}
\hspace{0.4cm}
		\begin{tikzpicture}[scale=0.7, every node/.style={circle, fill, inner sep=1.2pt}]
	
	\node[fill=black] (v1)at(-0.8,-1) {};
	\node[fill=black] (v2) at (-0.9, -0.5) {};
	\node[fill=black] (v3) at (-1, 0) {};
	\node[fill=black] (v4) at (-0.5, 0.5) {};
	\node[fill=black] (v5) at (0, 1) {};
	\node[fill=black] (v6) at (0.5, 0.5) {};
        \node[fill=black] (v7) at (1, 0) {};
        \node[fill=black] (v8) at (0.9, -0.5) {};
        \node[fill=black] (v9) at (0.8, -1) {};
        \node[fill=black] (v10) at (0, -1) {};
	\draw[black, thick] (v1) -- (v2);
        \draw[black, thick] (v2) -- (v3);
        \draw[black, thick] (v3) -- (v4);
        \draw[black, thick] (v4) -- (v5);
        \draw[black, thick] (v5) -- (v6);
        \draw[black, thick] (v6) -- (v7);
        \draw[black, thick] (v7) -- (v8);
        \draw[black, thick] (v8) -- (v9);
        \draw[black, thick] (v8) -- (v9);
        \draw[black, thick] (v10) -- (v9);
        \draw[black, thick] (v10) -- (v1);
	\node[draw=none,fill=none] at (0, -2) {$C_5$};
\end{tikzpicture}
\end{figure}

%\textcolor{red}{picture $ B_4,P_4$ et.al. and the bound.\\ What's Berge Path\\ cycle is more difficult than paths and different from Turan problems in 2-graphs\\}

A \emphd{Berge path} of length $k$ in a hypergraph consists of $k$ hyperedges $h_1, h_2,...,h_k$ such that $\{v_i,v_{i+1}\}\subseteq h_i$ for all $i \in [k]$, where $v_1, v_2,...,v_{k+1}$ are $k+1$ distinct vertices. In~\cite{MR4835949}, Gy\H{o}ri et al. showed that an $n$-vertex linear 3-graph without a Berge path of length $k\ge 4$ as a subgraph has at most $\frac{k-1}{6}n$ edges. Note that there are more than one 3-uniform Berge Path with length $n$ and some of them is not acyclic. A \emphd{linear path} of length $k$, denoted by $P_k$, is the only acyclic Berge path, which has  $2k+1$ vertices.  In the same paper, Gy\H{o}ri et al. asked the linear Tur\'an number of {linear paths} with length larger than 4 (see~\Cref{conj1.1}, or Conjecture 3.4 in~\cite{MR4835949}).

\begin{conj}\label{conj1.1}
	Let $G$ be an $n$ vertex linear 3-graph, containing no linear
	path of length $k \ge 5$. Then the number of edges in $G$ is at most $\frac{k}{3}n+cn$, for some universal constant c.
\end{conj}

Note that the bound in~\Cref{conj1.1} is asymptotic sharp for infinitely many pairs of {$(n,k)$}. A \ti{maximal partial triple system}, $MPTS(m)$, is a linear triple system on m points whose triples cover the maximum number of pairs of points.  Note that the disjoint union of $MPTS(2k)$ does not contain a path of length $k$ (which has $2k+1$ vertices), and it has roughly $\frac{(k-1)}{3}n$. Particularly, if $m\equiv 1,3 \mod{6}$, it is known as Steiner Triple System, $STS(m)$, which has $\frac{(m-1)m}{6}$. It is also known that an $MPTS(m)$ has exactly $\frac{(m-1)m-8}{6}$ edges for $m\equiv 5 \mod{6}$, $\frac{(m-2)m}{6}$ edges for $m\equiv 0,2 \mod{6}$,  and $\frac{(m-2)m-2}{6}$ edges for $m\equiv 4 \mod{6}$.

The main result of this paper is to give the linear Turan number of linear path with length 5, for which the extremal graph is not disjoint union of maximal partial triple system $MPTS(10)$. The later graph has $\frac{13}{10}n$ edges.

\begin{thm}\label{main}
Let $G$ be an $n$ vertex linear 3-graph, containing no $P_5$. Then the number of edges in $G$ is at most $\frac{15}{11}n$, and the equality holds if and only if the graph is the disjoint union of $G_0$, a graph with 11 vertices and 15 edges as shown below.
\end{thm}

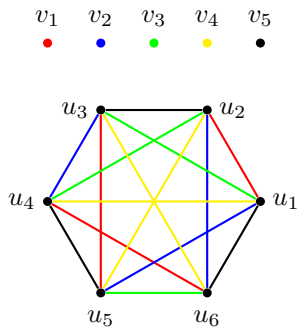
\begin{minipage}[htbp]{0.45\linewidth}
\begin{figure}[H]
	\centering
\begin{tikzpicture}[scale=0.7, every node/.style={circle, fill, inner sep=1.2pt}]
    % 顶部的五个V_，挪动到更高的位置
    \node[fill=red, label=above:$v_1$] (v1) at (-2, 3) {};
    \node[fill=blue, label=above:$v_2$] (v2) at (-1, 3) {};
    \node[fill=green, label=above:$v_3$] (v3) at (0, 3) {};
    \node[fill=yellow, label=above:$v_4$] (v4) at (1, 3) {};
    \node[fill=black, label=above:$v_5$] (v5) at (2, 3) {};

    % 圆上的六个V_
    \node[label=right:$u_1$] (u1) at (0:2cm) {};
    \node[label=right:$u_2$] (u2) at (60:2cm) {};
    \node[label=left:$u_3$] (u3) at (120:2cm) {};
    \node[label=left:$u_4$] (u4) at (180:2cm) {};
    \node[label=below:$u_5$] (u5) at (240:2cm) {};
    \node[label=below:$u_6$] (u6) at (300:2cm) {};

    % 连接边，使用不同颜色表示不同的vV_控制的边
    \draw[red, thick] (u1) -- (u2);
    \draw[red, thick] (u3) -- (u5);
    \draw[red, thick] (u4) -- (u6);
    
    \draw[blue, thick] (u1) -- (u5);
    \draw[blue, thick] (u3) -- (u4);
    \draw[blue, thick] (u2) -- (u6);

    \draw[green, thick] (u1) -- (u3);
    \draw[green, thick] (u2) -- (u4);
    \draw[green, thick] (u5) -- (u6);

    \draw[yellow, thick] (u1) -- (u4);
    \draw[yellow, thick] (u2) -- (u5);
    \draw[yellow, thick] (u3) -- (u6);

    \draw[black, thick] (u1) -- (u6);
    \draw[black, thick] (u2) -- (u3);
    \draw[black, thick] (u4) -- (u5);
\end{tikzpicture}
\caption{extremal graph $G_0$}
\label{fextremal}
\end{figure}
\end{minipage}
\hfill
\begin{minipage}[htbp]{0.55\linewidth}
\begin{tabular}[t]{c|c}
\hline
Vertices &  $v_1,v_2,\ldots,v_5,u_1,u_2,\ldots,u_6$ \\
\hline
 & $\{v_1,u_1,u_2\},\{v_1,u_3,u_5\},\{v_1,u_4,u_6\},$ \\
 & $\{v_2,u_1,u_5\},\{v_2,u_2,u_6\},\{v_2,u_3,u_4\},$ \\
Edges & $\{v_3,u_1,u_3\},\{v_3,u_2,u_4\},\{v_3,u_5,u_6\},$ \\
 &   $\{v_4,u_1,u_4\},\{v_4,u_2,u_5\},\{v_4,u_3,u_6\},$ \\
 &   $\{v_5,u_1,u_6\},\{v_5,u_2,u_3\},\{v_5,u_4,u_5\}.$ \\
\hline
\end{tabular}
\end{minipage}

\section{Notations and Sketch of Proof}

Given a graph $G$, the neighborhood of $v$ in $G$ is denoted by $N_G(v)$, the degree of $v$ in $G$ is denoted by $d_G(v)$, the minimum degree of $G$ is denoted by $\delta(G)$, the subgraph induced by $S\subseteq V(G)$ is denoted by $G[S]$. For two set $A$ and $B$, we write $A-B$ for the set difference $A\setminus B$. For any integer $n$, we denote by $[n]$ the set of integers $\{1,2,\ldots,n\}$. For simplicity, if no ambiguity, we identify a graph $G$ with its vertex set, a single vertex set $\{v\}$ with the element $v$ and we omit a vertex pair $\{u,v\}$ as $uv$; a 3-hyperedge $\{u,v,w\}$ as $uvw$. 

A \emphd{linear path of length $k$}, denoted by $P_k$, is the graph with $2k+1$ distinct vertices $\{u_i\mid 1\le i \le 2k+1\}$ and edge set $E = \{e_i=u_{2j-1}u_{2j}u_{2j+1}\mid 1\le j \le k\}.$ We also denote it by the sequence of vertices $P_k=u_1u_2\ldots u_{2k+1}$ or by the sequence of edges $P_k=e_1,e_2,\ldots,e_k$.
Similarly, a \emphd{linear cycle of length $k$}, denoted by $C_k$, is the graph with $2k$ vertices $\{u_i\mid 1\le i \le 2k+1\}$ and edge set $E = \{u_{2j-1}u_{2j}u_{2j+1}\mid 1\le j \le k\}$, where $u_{2k+1}=u_1$. 
A \emphd{linear star with $k$ edges}, denoted by $S_k$, is the graph with $2k+1$ vertices $\{u\}\cup\{v_i\}_{1\le i \le 2k}$ and edge set $E = \{uv_{2j-1}v_{2j}\mid 1\le j \le k\}.$ A non-trivial star is a star of at least two edges. In this paper, when we talk about paths, cycles and stars, we refer to linear paths, cycles and stars.

Given a subset $S\subseteq V(G)$, let $E_i(S) = \{e\in E (G)||e\cap S|=i\},$ for $1\le i\le 3$, and $E(S) =E_1(S) \cup E_2(S)\cup E_3(S)$, the set of all the edges incident to vertices in $S$. Note that $E(v) = E_1(v) = d_G(v)$ for a single vertex $v$, and one should be careful with the difference between $E_3(S)$ and $E(S)$.

In the proof, we will encounter some small graphs with at most 4 edges. To describe them, we use the same notations of linear 3-graph introduced by Colbourn and Dinitz~\cite{ea299c3877e64a5b996676debaabb8d3} as listed in~\Cref{fnotation}, and change the name of some graphs to avoid confusion.

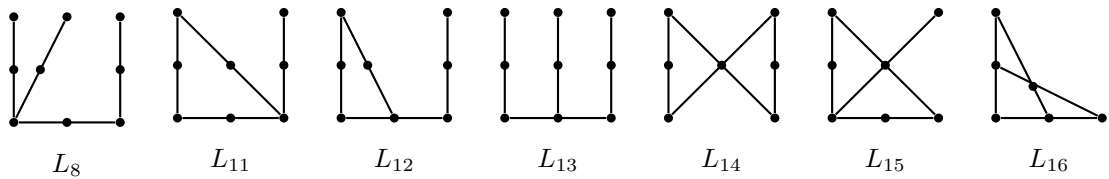
\begin{figure}[h]
	\centering
	\begin{tikzpicture}[scale=0.7, every node/.style={circle, fill, inner sep=1.2pt}]
		\node[fill=black] (v1) at (-1, 0) {};
		\node[fill=black] (v2) at (0, 0) {};
		\node[fill=black] (v3) at (1, 0) {};
		\node[fill=black] (v4) at (-1, 1) {};
		\node[fill=black] (v5) at (-1, 2) {};
		\node[fill=black] (v6) at (-0.5, 1) {};
		\node[fill=black] (v7) at (0, 2) {};
		\node[fill=black] (v8) at (1, 1) {};
		\node[fill=black] (v9) at (1, 2) {};
		
		\draw[black, thick] (v1) -- (v3);
		\draw[black, thick] (v1) -- (v5);
		\draw[black, thick] (v1) -- (v7);
		\draw[black, thick] (v3) -- (v9);
		\node[draw=none,fill=none] at (0, -0.8) {$L_8$};
	\end{tikzpicture}
	\hspace{0.4cm}
	\begin{tikzpicture}[scale=0.7, every node/.style={circle, fill, inner sep=1.2pt}]
		\node[fill=black] (v1) at (-1, 0) {};
		\node[fill=black] (v2) at (0, 0) {};
		\node[fill=black] (v3) at (1, 0) {};
		\node[fill=black] (v4) at (-1, 1) {};
		\node[fill=black] (v5) at (-1, 2) {};
		\node[fill=black] (v6) at (0, 1) {};
		\node[fill=black] (v8) at (1, 1) {};
		\node[fill=black] (v9) at (1, 2) {};
		
		\draw[black, thick] (v1) -- (v3);
		\draw[black, thick] (v1) -- (v5);
		\draw[black, thick] (v3) -- (v5);
		\draw[black, thick] (v3) -- (v9);
		\node[draw=none,fill=none] at (0, -0.8) {$L_{11}$};
	\end{tikzpicture}
	\hspace{0.4cm}
	\begin{tikzpicture}[scale=0.7, every node/.style={circle, fill, inner sep=1.2pt}]
		\node[fill=black] (v1) at (-1, 0) {};
		\node[fill=black] (v2) at (0, 0) {};
		\node[fill=black] (v3) at (1, 0) {};
		\node[fill=black] (v4) at (-1, 1) {};
		\node[fill=black] (v5) at (-1, 2) {};
		\node[fill=black] (v6) at (1, 1) {};
		\node[fill=black] (v7) at (1, 2) {};
		\node[fill=black] (v8) at (-0.5, 1) {};
		
		\draw[black, thick] (v1) -- (v3);
		\draw[black, thick] (v1) -- (v5);
		\draw[black, thick] (v3) -- (v7);
		\draw[black, thick] (v5) -- (v2);
		\node[draw=none,fill=none] at (0, -0.8) {$L_{12}$};
	\end{tikzpicture}
	\hspace{0.4cm}
	\begin{tikzpicture}[scale=0.7, every node/.style={circle, fill, inner sep=1.2pt}]
		% 顶部的五个V_，挪动到更高的位置
		\node[fill=black] (v1) at (-1, 0) {};
		\node[fill=black] (v2) at (0, 0) {};
		\node[fill=black] (v3) at (1, 0) {};
		\node[fill=black] (v4) at (-1, 1) {};
		\node[fill=black] (v5) at (-1, 2) {};
		\node[fill=black] (v6) at (0, 1) {};
		\node[fill=black] (v7) at (0, 2) {};
		\node[fill=black] (v8) at (1, 1) {};
		\node[fill=black] (v9) at (1, 2) {};
		
		% 连接边，使用不同颜色表示不同的vV_控制的边
		\draw[black, thick] (v1) -- (v3);
		\draw[black, thick] (v1) -- (v5);
		\draw[black, thick] (v2) -- (v7);
		\draw[black, thick] (v3) -- (v9);
		\node[draw=none,fill=none] at (0, -0.8) {$L_{13}$};
	\end{tikzpicture}
	\hspace{0.4cm}
	\begin{tikzpicture}[scale=0.7, every node/.style={circle, fill, inner sep=1.2pt}]
		\node[fill=black] (v1) at (-1, 0) {};
		\node[fill=black] (v2) at (0, 1) {};
		\node[fill=black] (v3) at (1, 0) {};
		\node[fill=black] (v4) at (-1, 1) {};
		\node[fill=black] (v5) at (-1, 2) {};
		\node[fill=black] (v6) at (1, 1) {};
		\node[fill=black] (v7) at (1, 2) {};	
		
		\draw[black, thick] (v1) -- (v7);
		\draw[black, thick] (v1) -- (v5);
		\draw[black, thick] (v3) -- (v5);
		\draw[black, thick] (v3) -- (v7);
		\node[draw=none,fill=none] at (0, -0.8) {$L_{14}$};
	\end{tikzpicture}
	\hspace{0.4cm}
	\begin{tikzpicture}[scale=0.7, every node/.style={circle, fill, inner sep=1.2pt}]
		\node[fill=black] (v1) at (-1, 0) {};
		\node[fill=black] (v2) at (0, 0) {};
		\node[fill=black] (v3) at (1, 0) {};
		\node[fill=black] (v4) at (-1, 1) {};
		\node[fill=black] (v5) at (-1, 2) {};
		\node[fill=black] (v6) at (0, 1) {};
		\node[fill=black] (v9) at (1, 2) {};
		
		\draw[black, thick] (v1) -- (v3);
		\draw[black, thick] (v3) -- (v5);
		\draw[black, thick] (v1) -- (v9);
		\draw[black, thick] (v1) -- (v5);
		\node[draw=none,fill=none] at (0, -0.8) {$L_{15}$};
	\end{tikzpicture}
	\hspace{0.4cm}
	\begin{tikzpicture}[scale=0.7, every node/.style={circle, fill, inner sep=1.2pt}]
		\node[fill=black] (v1) at (-1, 0) {};
		\node[fill=black] (v2) at (0, 0) {};
		\node[fill=black] (v3) at (1, 0) {};
		\node[fill=black] (v4) at (-1, 1) {};
		\node[fill=black] (v5) at (-1, 2) {};
		\node[fill=black] (v6) at (-0.3, 0.6) {};
		
		\draw[black, thick] (v1) -- (v3);
		\draw[black, thick] (v1) -- (v5);
		\draw[black, thick] (v3) -- (v4);
		\draw[black, thick] (v2) -- (v5);
		\node[draw=none,fill=none] at (0, -0.8) {$L_{16}$};
	\end{tikzpicture}
	\caption{Notations for small linear 3-graphs}
	\label{fnotation}
\end{figure}

Now we give a sketch of the proof of~\Cref{main}. Note that we just need to consider the case that $G$ is connected. Among the connected linear 3-graphs that do not contain $P_5$ as a subgraph, if there exist a graph $G$ with more than $15n(G)/11$ edges, let $G$ be the one with the minimum number of vertices. Otherwise, let $G$ be one with exactly $15n(G)/11$ edges. We are going to prove that $G$ must be $G_0$. Note that if $n<11$, we always have $E (G)<15n/11$ (See Steiner triple systems in~\cite{ea299c3877e64a5b996676debaabb8d3}), a contradiction. So we can assume $n\geq 11$.

We start with a lemma which we will be used repeatedly in the future proof.

\begin{lem}\label{contradiction}
    For any vertex set $S\subseteq G$, $|E(S)| \ge \frac{15}{11}|S|$. In particular, $\delta(G)\ge 2$.
\end{lem}
\begin{proof}If $S = V(G)$, then it is clear by definition of $G$. Otherwise, we delete $S$ from $G$, thereby deleting $|S|$ vertices and less than $15|S|/11$ edges, obtaining a smaller counterexample with number of edges strictly larger than $15n/11$, which is not isomorphic to $G_0$, contradicting the minimality of $G$.
\end{proof}

The rest of the proof is divided into 2 parts. 
In~\Cref{section3}, we prove that there are no two disjoint $P_2$ in $G$. To prove it, we show that if $G$ has two specified disjoint structures but does not contain $P_5$, then we can find some vertex set $S$ with not many edges incident to it, which contradict to Lemma~\Cref{contradiction}. 
Base on this, in~\Cref{Section4}, by considering the ``link graph'' of an edge, we prove the main theorem. 

\section{The Case With Two Disjoint $P_2$}\label{section3}

In this section, we prove the following lemma.

\begin{lem}\label{S2S2}
    There are no two disjoint $P_2$ in $G$.
\end{lem}

We begin with some preparations. 
\begin{defi}
    Let $U$ be a subgraph of $G$.
    \begin{enumerate}
        \item A vertex $u_0$ is called a \emphd{$k$-center} of $U$ if there is no path starting from $u_0$ with length larger than $k$.
        \item A pair of vertices $\{u_1,u_2\}\subseteq U$ is called a \emphd{center pair} of $U$ if they are not adjacent in $U$ and any edge in $U$ is incident to one of them.
        \item The maximum number of \emphd{disjoint} center pairs in $U$ is denoted by $\gamma(U)$.
    \end{enumerate}   
     Note that only star has 1-center, which is its center. We always abbreviate 1-center to \emphd{center}. 
\end{defi}

\begin{exmp}\label{exmps}
	A single edge is a graph with three 1-centers and no center pairs; $S_3$ is a graph with no center pairs and seven 2-centers; $C_4$ is a graph with one center pair and no 2-centers; $P_4$ is a graph with one 2-center and one center pair (See~\Cref{fsp}, the red vertices are 2-centers and the red pairs are center pairs). 
\end{exmp}
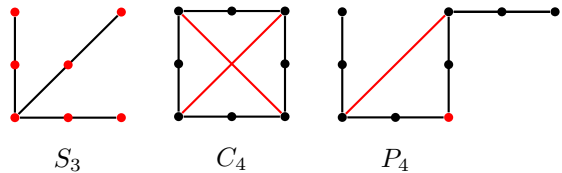
\begin{figure}[h]
	\centering
	\begin{tikzpicture}[scale=0.7, every node/.style={circle, fill, inner sep=1.2pt}]
		\node[fill=red] (v1) at (-1, 0) {};
		\node[fill=red] (v3) at (1, 0) {};
		\node[fill=red] (v5) at (-1, 2) {};
		\node[fill=red] (v7) at (1, 2) {};
		\draw[black, thick] (v1) -- (v3);
		\draw[black, thick] (v1) -- (v5);
		\draw[black, thick] (v1) -- (v7);
        \node[fill=red] (v2) at (0, 0) {};
        \node[fill=red] (v6) at (0, 1) {};
		\node[fill=red] (v4) at (-1, 1) {};
		\node[draw=none,fill=none] at (0, -0.8) {$S_3$};
	\end{tikzpicture}
	\hspace{0.4cm}
	\begin{tikzpicture}[scale=0.7, every node/.style={circle, fill, inner sep=1.2pt}]
	
	\node[fill=black] (v1) at (-1, 0) {};
	\node[fill=black] (v2) at (0, 0) {};
	\node[fill=black] (v3) at (1, 0) {};
	\node[fill=black] (v4) at (-1, 1) {};
	\node[fill=black] (v5) at (-1, 2) {};
	\node[fill=black] (v6) at (1, 1) {};
	\node[fill=black] (v7) at (1, 2) {};
	\node[fill=black] (v8) at (0, 2) {};
	
	\draw[black, thick] (v1) -- (v3);
	\draw[black, thick] (v7) -- (v5);
	\draw[black, thick] (v3) -- (v7);
	\draw[black, thick] (v1) -- (v5);
	\draw[red, thick] (v3) -- (v5);
	\draw[red, thick] (v1) -- (v7);
	\node[draw=none,fill=none] at (0, -0.8) {$C_4$};
\end{tikzpicture}
	\hspace{0.4cm}
	\begin{tikzpicture}[scale=0.7, every node/.style={circle, fill, inner sep=1.2pt}]
		\node[fill=black] (v1) at (-1, 0) {};
		\node[fill=black] (v2) at (0, 0) {};
		\node[fill=red] (v3) at (1, 0) {};
		\node[fill=black] (v4) at (-1, 1) {};
		\node[fill=black] (v5) at (-1, 2) {};
		\node[fill=black] (v6) at (1, 1) {};
		\node[fill=black] (v7) at (1, 2) {};
		\node[fill=black] (v8) at (2, 2) {};
		\node[fill=black] (v9) at (3, 2) {};
		
		\draw[black, thick] (v1) -- (v3);
		\draw[black, thick] (v9) -- (v7);
		\draw[black, thick] (v3) -- (v7);
		\draw[black, thick] (v1) -- (v5);
		\draw[black, thick] (v7) -- (v9);
		\draw[red, thick] (v1) -- (v7);
		\node[draw=none,fill=none] at (0, -0.8) {$P_4$};
	\end{tikzpicture}
	\caption{$S_3$ and $P_4$}
	\label{fsp}
\end{figure}
\begin{prop}\label{gammaUleq3}
	If $U$ is a connected graph, then $\gamma(U)\le 3$. There are two connected subgraphs for which $\gamma(U)=3$ and three subgraphs with $\gamma(U)=2$.
\end{prop}
\begin{proof}
     Every edge can intersect at most three disjoint center pairs, so $\gamma(U)\le 3$.
     If $\gamma(U)\geq 2$, then it has at most 4 edges as it is a linear graph. Combined with the list of all 3-graph with less than four edges in~\cite{ea299c3877e64a5b996676debaabb8d3}, it's not hard to draw all graphs with $\gamma(U)=2$ or $3$ as in~\Cref{fgamma}, where the red pairs are center pairs.
\end{proof}
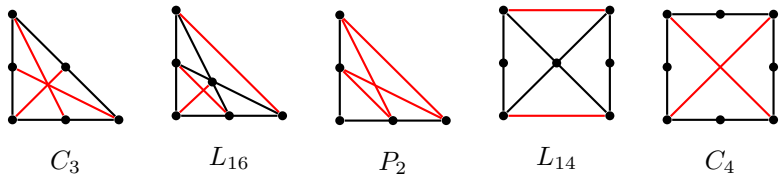
\begin{figure}[h]
	\centering
	\begin{tikzpicture}[scale=0.7, every node/.style={circle, fill, inner sep=1.2pt}]
		\node[fill=black] (v1) at (-1, 0) {};
		\node[fill=black] (v2) at (0, 0) {};
		\node[fill=black] (v3) at (1, 0) {};
		\node[fill=black] (v4) at (-1, 1) {};
		\node[fill=black] (v5) at (-1, 2) {};
		\node[fill=black] (v6) at (0, 1) {};
		
		\draw[black, thick] (v1) -- (v3);
		\draw[black, thick] (v1) -- (v5);
		\draw[black, thick] (v3) -- (v5);
		\draw[red, thick] (v1) -- (v6);
		\draw[red, thick] (v3) -- (v4);
		\draw[red, thick] (v2) -- (v5);
		\node[draw=none,fill=none] at (0, -0.8) {$C_3$};
	\end{tikzpicture}
	\hspace{0.4cm}
	\begin{tikzpicture}[scale=0.7, every node/.style={circle, fill, inner sep=1.2pt}]
		\node[fill=black] (v1) at (-1, 0) {};
		\node[fill=black] (v2) at (0, 0) {};
		\node[fill=black] (v3) at (1, 0) {};
		\node[fill=black] (v4) at (-1, 1) {};
		\node[fill=black] (v5) at (-1, 2) {};
		\node[fill=black] (v6) at (-0.32, 0.64) {};

		\draw[black, thick] (v1) -- (v3);
		\draw[black, thick] (v1) -- (v5);
		\draw[black, thick] (v3) -- (v4);
		\draw[black, thick] (v2) -- (v5);
		\draw[red, thick] (v1) -- (v6);
		\draw[red, thick] (v2) -- (v4);
		\draw[red, thick] (v3) -- (v5);
		\node[draw=none,fill=none] at (0, -0.8) {$L_{16}$};
	\end{tikzpicture}
	\hspace{0.4cm}
	\begin{tikzpicture}[scale=0.7, every node/.style={circle, fill, inner sep=1.2pt}]
		\node[fill=black] (v1) at (-1, 0) {};
		\node[fill=black] (v2) at (0, 0) {};
		\node[fill=black] (v3) at (1, 0) {};
		\node[fill=black] (v4) at (-1, 1) {};
		\node[fill=black] (v5) at (-1, 2) {};
		
		\draw[black, thick] (v1) -- (v3);
		\draw[black, thick] (v1) -- (v5);
		\draw[red, thick] (v3) -- (v5);
		\draw[red, thick] (v3) -- (v4);
		\draw[red, thick] (v2) -- (v5);
		\draw[red, thick] (v2) -- (v4);
		\node[draw=none,fill=none] at (0, -0.8) {$P_{2}$};
	\end{tikzpicture}
	\hspace{0.4cm}
	\begin{tikzpicture}[scale=0.7, every node/.style={circle, fill, inner sep=1.2pt}]
		\node[fill=black] (v1) at (-1, 0) {};
		\node[fill=black] (v2) at (0, 1) {};
		\node[fill=black] (v3) at (1, 0) {};
		\node[fill=black] (v4) at (-1, 1) {};
		\node[fill=black] (v5) at (-1, 2) {};
		\node[fill=black] (v6) at (1, 1) {};
		\node[fill=black] (v7) at (1, 2) {};	
		
		\draw[black, thick] (v1) -- (v7);
		\draw[black, thick] (v1) -- (v5);
		\draw[black, thick] (v3) -- (v5);
		\draw[black, thick] (v3) -- (v7);
		\draw[red, thick] (v1) -- (v3);
		\draw[red, thick] (v5) -- (v7);
		\node[draw=none,fill=none] at (0, -0.8) {$L_{14}$};

	\end{tikzpicture}
	\hspace{0.4cm}
	\begin{tikzpicture}[scale=0.7, every node/.style={circle, fill, inner sep=1.2pt}]

		\node[fill=black] (v1) at (-1, 0) {};
		\node[fill=black] (v2) at (0, 0) {};
		\node[fill=black] (v3) at (1, 0) {};
		\node[fill=black] (v4) at (-1, 1) {};
		\node[fill=black] (v5) at (-1, 2) {};
		\node[fill=black] (v6) at (1, 1) {};
		\node[fill=black] (v7) at (1, 2) {};
		\node[fill=black] (v8) at (0, 2) {};
		
		\draw[black, thick] (v1) -- (v3);
		\draw[black, thick] (v7) -- (v5);
		\draw[black, thick] (v3) -- (v7);
		\draw[black, thick] (v1) -- (v5);
		\draw[red, thick] (v3) -- (v5);
		\draw[red, thick] (v1) -- (v7);
		\node[draw=none,fill=none] at (0, -0.8) {$C_4$};
	\end{tikzpicture}
	\caption{All graphs with $\gamma=2,3$}
	\label{fgamma}
\end{figure}
\begin{defi}
    For three disjoint vertex subsets $U, V, W$ of $G$, we say a path $P$ connects $U$ and $V$ (by $W$), if $P = v_1v_2\ldots v_{2k+1}$ is a path of length $k$ in $G$, such that $\{v_1,\ldots,v_i\}\subseteq U$ with $i=1$ or $2$, $\{v_{j},\ldots,v_{2k+1}\}\subseteq V$ with $j=2k$ or $2k+1$, and $W=\{v_{i+1},\ldots,v_{j-1}\}$ ($W$ can be $\emptyset$). Similarly we define an edge $e$ connects $U$ and $V$, by considering it as a path of length 1.
\end{defi}

\begin{lem}\label{conpath}
    Let $U,V$ be two disjoint connected subgraphs of $G$, and $P$ is a path connecting $U$ and $V$ which does not contain a center of $U$ or $V$. Then one of the following holds.
    \begin{enumerate}
    \item $P$ is an edge and consists of a 2-center of one of $U,V$ and a center pair of the other.
    \item $P$ is a path of length 2 and $P\cap U, P\cap V$ are center pairs of $U,V$ respectively.
    \end{enumerate}
\end{lem}

\begin{proof}
    For any connected subgraph $X$ of $G$ and edge $g$ with $|g\cap X|\leq 2$, by connectivity it is not hard to see that
    \begin{enumerate}
        \item If $|g\cap X|=2$ and $g\cap U$ is not a center pair of $U$, then there is a $P_3$ starting from $g$ in $U\cup g$.
        \item If $|g\cap X|=1$ and $g\cap U$ is not a $k$-center of $U$, then there is a $P_{k+2}$ starting from $g$ in $U\cup g$.
    \end{enumerate}

    First assume $P$ is an edge. Since $G$ is 3-uniform, obviously at least one of $|P\cap U|$ and $|P\cap V|$ equals 1 and so we can assume $|P\cap U|=1$. Since $P$ does not contain a center of $U$, there is a $P_3$ starting from $P$ in $U\cup e$ by the claim. If $|P\cap V|=1$, similarly there exist another $P_3$ starting from $e$ in $V\cup e$ by the claim.
    These two $P_3$ together form a $P_5$ in $G$, a contradiction. Therefore, $|P\cap V|=2$ and if $P\cap V$ is not a center pair, again by the claim we obtain a $P_3$ in $V\cup e$ and further a $P_5$ in $G$, a contradiction.
     
    If $P$ has length at least 2, note that $P$ can be extended by at least one edge in $U$ and another in $V$. So to avoid the presence of $P_5$, $P$ has length 2. Let $P = e_1e_2$. Then $e_1$ connects $U$ and $V\cup e_2$ and does not contain a center of them. By the previous discussion, $e_1\cap U$ is a center pair of $U$ and by symmetry, $e_2\cap V$ is a center pair of $V$.
\end{proof}

\begin{lem}\label{lmcon}
$G$ is 2-connected. That is, after deleting any vertex in $G$, the remaining graph is still connected.
\end{lem}
\begin{proof}
    If not, let $u$ be a cut vertex. For any component $C$ in $G-u$, we identify $G[C]$ and $C$ in the proof of this lemma.

    Suppose there is a component $S$ which is a star centered at $v$. Note that it is not an isolated vertex or a single edge, otherwise $|E(S)|\leq \max\{\frac{1}{1},\frac{3}{2}\}|S|<\frac{15}{11}|S|$, contradicting~\Cref{contradiction}. Let $S=\{v,w_1,w_2,\ldots,w_{2k}\}$ such that $E_3(S)=\{vw_{2i-1}w_{2i}\mid 1\leq i \leq k\}$, $k\geq 2$.
    If $uw_1w_3\in E (G)$, then by linearity $E(\{w_1,w_2,w_3,w_4\})\leq 5<\frac{15}{11}\cdot 4$, contradicting~\Cref{contradiction}. By symmetry we can assume there is no edge $e$ incident to $u$ such that $|e\cap S|=2$. However, since $d(w_i)\ge 2$ and $u$ is a cut vertex, there exists $v_i\in G$ such that $uw_iv_i\in E (G)$, $1\leq i\leq k$.
    By previous discussion, $v_i\notin S$ and thus there is a $P_3$ in $S\cup \{u,v_i\}$ starting from $v_i$. To avoid the presence of $P_5$, $v_i$ is the center of the component containing $v_i$ and thus the component is a star. Since there is exactly one edge between $u$ and the center of each star components, it follows that the number of centers of star components is at least the number of leaves, which is absurd since each star has at least 4 leaves. Thus there exists no component $S$ which is a star.

    Fix two different components $D_1,D_2$, by~\Cref{conpath} there is no edge $e=\{v_1,u,v_2\}$ such that $v_1,v_2$ are in $D_1,D_2$. But $G$ is connected, so there exists an edge $\{u,v_1,v_2\}$ such that $v_1,v_2\in D_1$.
    Again by~\Cref{conpath} applied to $D_1$ and $D_1\cup u$, we have $u$ is a 2-center of $D_1\cup u$ and $\{v_1,v_2\}$ is a center pair of $D_2$. That is, neighbors of $u$ in $D_1$ are disjoint center pairs. Let $T$ be the neighbors of $u$ in $C$. If $T\geq 4$, then $\gamma(D_1)\geq 2$ and by~\Cref{gammaUleq3} we have $E(D_1)\leq \max\{\frac{7}{6},\frac{4}{5},\frac{6}{7},\frac{6}{8}\}|D_1|<\frac{15}{11}|D_1|$, contradicting~\Cref{contradiction}. Thus $T=2$ and each edge in $D_1$ is incident to $T$. That is, each vertex in $D_1-T$ has degree 2 and $E(D_1-T)\leq |D_1-T|$, contradicting~\Cref{contradiction}.
\end{proof}

Now we begin to prove~\Cref{S2S2}. We will show that after deleting the vertices of any $P_2$ in $G$, the remaining graph cannot contain the following structures step by step:\\

\emphd{Case 1.} there is a $P_4$ or $C_4$ in the remaining graph (\Cref{4P2}).\\

\emphd{Case 2.} there is a $P_3$ in the remaining graph and it is not a component (\Cref{I3P2}).\\

\emphd{Case 3.} there is a $P_3$ or $C_3$ in the remaining graph (\Cref{3P2}).\\

\emphd{Case 4.} there is a $P_2$ in the remaining graph (\Cref{2P2}).\\

The key idea is to consider the 2-centers and center pairs in the disjoint structures. For readability, in this section we will color 2-centers and center pairs red.

\begin{lem}\label{4P2}
	For $U=P_4$ or $C_4$, there is no disjoint $U$ and $P_2$ in $G$.
\end{lem}
\begin{proof}
    Denote the number of center pairs and 2-centers in $U$ by $a(U)$ and $b(U)$ ($a(U)$ is different from $\gamma(U)$, the maximum number of disjoint center pairs). Note that $a(P_4)=b(P_4)=1$, $a(C_4)=2, b(C_4)=0$ and the two center pairs in $C_4$ are disjoint.
	Suppose the lemma is false, then there exist a $P_4$ or $C_4$, denoted by $U$, such that there is a component $V$ in $G-U$ with at least 2 edges. Note that any edge $e$ in $E(V)-E_3(V)$ must intersect $U$, and thus connects $U$ and $V$. By~\Cref{conpath}, $e$ consists of 2-centers and center pairs of $U,V$. Since $G$ is linear, each center pair in $U$ has at most one neighbor in $U$, and each 2-center in $U$ is adjacent to at most $\gamma(V)$ center pairs in $V$. Thus $|E(V)-E_3(V)|\le a(U)+\gamma(V)b(U)\leq 4$.
	
	Suppose $V$ is a $k$-star with $2k$ leaves, $k\ge 2$.
    Note that $\gamma(S_4)=2$ and $\gamma(S_{2k})=0$ for $k\geq 3$.
    Let $L$ be the leaves of $V$ and we have $|E(L)|\le E_3(V)+a(U)+2b(U)\leq k+3<\frac{15}{11}|S|$, contradicting~\Cref{contradiction}.
	So we can assume $V$ is not a star and then discuss based on the value $\gamma(V)$.

    If $\gamma(V)\geq 1$, let $\{v_1,v_2\}$ be a center pair. Then any edge in $V$ is incident to $\{v_1,v_2\}$ and by linearity we have $E_3(V)\leq |V|-2$. Thus $|E(V)|\leq |V|+2<\frac{15}{11}|V|$ since $|V|\geq 5$.
	
	If $\gamma(V) = 0$, there is no center pair in $V$. Then by~\Cref{conpath} each edge connecting $U$ to $V$ must consists of one center pair in $U$ and a 2-center in $V$.
    By~\Cref{lmcon}, $G$ is 2-connected, and thus there are at least two edges $g_1,g_2$ connecting $U,V$ containing different 2-centers $v_1,v_2$ in $V$. It follows that $a(U)= 2$, $U=C_4$ and $g_1,g_2$ are the only edges between $U,V$. Suppose there is an edge $e_1$ in $V$ disjoint from $v_1$ and $v_2$, since $G-v_2$ is connected, there is a path $P$ from $e_1$ to $g_1$ of length at least 3. Together with one edge in $U$ and $g_2$, $P$ can be extended to a path of length at least 5, a contradiction. Thus every edge in $V$ is incident to one of $v_1$ or $v_2$. However, $V$ has no center pair, so $v_1, v_2$ have to be adjacent, Thus there is an edge $\{v_1,v_2,v_3\}$ in $V$. Since any edge in $V$ is incident to $v_1$ or $v_2$, by linearity we have $d_G(v_3)=1$, contradicting to the fact that $\delta(G) = 2$.
\end{proof}

\begin{lem}\label{I3P2}
	After deleting the vertices of an arbitrary $P_2$ in $G$, if some component contains a $P_3$, then this component is exactly a $P_3$. 
\end{lem}
\begin{proof}
Suppose after deleting vertices of some subgraph $V=P_2$ in $G$ consisting of two edges $v_1v_2v_3, v_3v_4v_5$, there is a component containing a $P_3$ but not a $P_3$ by itself. By~\Cref{4P2}, it does not contain a $P_4$ or $C_4$ and thus the component contains a subgraph $U=L_{14},L_{11},L_{12},L_{8}$ or $L_{13}$.
    
(1) $U=L_{14}$.
    
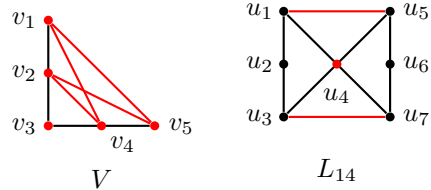
\begin{figure}[h]
	\centering
	\begin{tikzpicture}[scale=0.7, every node/.style={circle, fill, inner sep=1.2pt}]
		\node[fill=red] (v1) at (-1, 0) {};
		
		\node[fill=red] (v3) at (1, 0) {};
		
		\node[fill=red] (v5) at (-1, 2) {};

		\draw[black, thick] (v1) -- (v3);
		\draw[black, thick] (v1) -- (v5);

        \node[fill=red] (v4) at (-1, 1) {};
		\node[fill=red] (v2) at (0, 0) {};
        \draw[red, thick] (v3) -- (v4);
		\draw[red, thick] (v2) -- (v5);
		\draw[red, thick] (v3) -- (v5);
		\draw[red, thick] (v2) -- (v4);

		\node[draw=none,fill=none] at (-1.45, 2) {$v_1$};
		\node[draw=none,fill=none] at (-1.45, 1) {$v_2$};
		\node[draw=none,fill=none] at (-1.45, 0) {$v_3$};
		\node[draw=none,fill=none] at (0.4, -0.35) {$v_4$};
		\node[draw=none,fill=none] at (1.5, 0) {$v_5$};
		\node[draw=none,fill=none] at (0, -1) {$V$};
		
	\end{tikzpicture}%
    \hspace{0.4cm}
	\begin{tikzpicture}[scale=0.7, every node/.style={circle, fill, inner sep=1.2pt}]
	\node[fill=black] (v1) at (-1, 0) {};
	\node[fill=black] (v3) at (1, 0) {};
	\node[fill=black] (v4) at (-1, 1) {};
	\node[fill=black] (v5) at (-1, 2) {};
	\node[fill=black] (v6) at (1, 1) {};
	\node[fill=black] (v7) at (1, 2) {};	
	\node[draw=none,fill=none] at (-1.45, 2) {$u_1$};
	\node[draw=none,fill=none] at (-1.45, 1) {$u_2$};
	\node[draw=none,fill=none] at (-1.45, 0) {$u_3$};
	\node[draw=none,fill=none] at (0, 0.4) {$u_4$};
	\node[draw=none,fill=none] at (1.5, 2) {$u_5$};
	\node[draw=none,fill=none] at (1.5, 1) {$u_6$};
	\node[draw=none,fill=none] at (1.5, 0) {$u_7$};
	
	\draw[black, thick] (v1) -- (v7);
	\draw[black, thick] (v1) -- (v5);
	\draw[black, thick] (v3) -- (v5);
	\draw[black, thick] (v3) -- (v7);
		\node[fill=red] (v2) at (0, 1) {};
	\draw[red, thick] (v1) -- (v3);
	\draw[red, thick] (v5) -- (v7);
	\node[draw=none,fill=none] at (0, -1) {$L_{14}$};
	
\end{tikzpicture}	
\end{figure}

	In the figure above, we color the 2-centers red and connect the center pairs by red edges in $U$ and $V$. Let $S=V(U)-u_4$, $F_U$ be the 4 edges in $U$, $E^-(S) =E(S)-F_U-E(v_3)$ and $E_i^-(S) = E_i(S)-F_U-E(v_3)$ for $i=1,2,3$. Note that $E_3^-(S)=\emptyset$ by linearity, and thus $E^-(S) = E_1^-(S)\cup E_2^-(S)$. 
    
    By~\Cref{contradiction}, $|E(S)|\ge 9$, which means $|E^-(S)|+|E(v_3)\cap E(S)|\ge 5$. Our goal is to show that $|E^-(S)|+|E(v_3)\cap E(S)|\leq 4$ and obtain a contradiction. 

    For $g\in E_1^-(S)$, we claim that $u_4\in g$. If not, by symmetry we can assume $g\cap U\in\{u_1,u_2\}$. Then either $g,u_1u_2u_3,u_3u_4u_5,u_5u_6u_7$ form a $P_4$ disjoint from $V$, contradicting~\Cref{4P2},
    or there exists $g'\in\{v_1v_2v_3,v_3v_4v_5\}$ such that $g',g,u_1u_2u_3,u_3u_4u_5,u_5u_6u_7$ form a $P_5$ in $G$, a contradiction. Moreover, by linearity $u_1,u_3,u_5,u_7\notin g$, so $g\cap U\in \{u_2,u_6\}$. Since $g\notin E(v_3)$, if $g\cap V\neq \emptyset$, by symmetry we can assume $g=v_1u_2u_4$, 
    and then $v_5v_4v_3,v_3v_2v_1,v_1u_2u_4,u_4u_3u_5,u_5u_6u_7$ form a $P_5$ in $G$, a contradiction. In conclusion, for any $g\in E_1^-(S)$, $g=u_4u_2u_8$ or $u_4u_6u_9$, where $u_8$ and $u_9$ lie outside $U\cup V$. Since there is no $P_4$ disjoint from $V$, these two edges cannot exist simultaneously. Thus $|E_1^-(S)|\le 1$.

    For $g\in E(S)\cap E(v_3)$, by linearity $|g\cap V|=1$ and $g$ connects $V,U$. If $|g\cap U|=1$, since any vertex in $S$ is an endpoint of a $P_3$ in $U$, there exists a $P_5$ in $G$ consisting of this $P_3$ and $g, v_1v_2v_3$, a contradiction. Thus $|g\cap U|=2$, and by linearity we have $|E(S)\cap E(v_3)|\le \lfloor\frac{|U|}{2}\rfloor=3$.

     Let $E'$ be the set of edges in $E^-(S)$ containing a center pair of $U$, i.e., containing either $u_3u_7$ or $u_1u_5$.
    
    \emphd{Case 1.} $E'=\emptyset$. Since $|E(S)\cap E(v_3)|\le 3$, it suffices to show $|E_2^-(S)\cup E_1^-(S)|\le 1$.
    
    For $h \in E_2^-(S)$, by definition $v_3\notin h$, and by symmetry, we could suppose $h\cap S=u_1u_6$ or $u_2u_6$. Immediately we obtain that the third vertex $w= h-S\in U\cup V$ since otherwise $u_1u_2u_3,u_3u_4u_5,u_5u_6u_7,h$ is a $C_4$ disjoint from $V$, contradicting~\Cref{4P2}. If $w\in V$, then by definition, $w\neq v_3$, and by symmetry, we can assume $w=v_1$ and then $v_5v_4v_3,v_3v_2v_1,h,u_6u_5u_7,u_5u_4u_3$ forms a $P_5$, a contradiction. Thus, $w=u_4$ because $w\notin S$. By linearity, the only possible case of $h$ is $u_4u_2u_6$, which implies that $E_2^-(S)\subseteq\{u_4u_2u_6\}$. Again by linearity, any two of $u_2u_4u_6,u_2u_4u_8$ and $u_4u_6u_9$ cannot exist simultaneously, which implies that $|E_2^-(S)\cup E_1^-(S)|\le 1$. 

	\emphd{Case 2.} $|E'|=1$. By symmetry, we can assume $E'=\{u_1 u_5 w\}$. It suffices to show $|E_2^-(S)\cup E_1^-(S)|\le 1$.

    As the discussion in Case 1, $E_2^-(S)\subseteq E' \cup\{u_4u_2u_6\}$. If $u_4u_2u_6$ is an edge of $G$, then $w\notin U\cup V$, since otherwise $h,u_1u_2u_3,u_2u_4u_6,u_6u_7u_5$ is a $C_4$ disjoint from $V$, a contradiction. Further as in Case 1, we have $E_1^-(S)=\emptyset$, and thus $|E_2^-(S)\cup E_1^-(S)|\le 1$. So we can assume $E_2^-(S)=E'$.
    
    If $w\in U\cup V$, then by definition $w\in V-v_3$, and by symmetry we can assume $w=v_1$. If there is an edge $g\in E_1^-(S)$, by symmetry we can assume $g=u_2u_4w'$ such that $w'\notin U\cup V$. Then $v_5v_4v_3,v_3v_2v_1,h,u_1u_2u_3,g$ is a $P_5$ in $G$, a contradiction.
    
    If $w\notin U\cup V$, we claim that there exists no edge $g=u_2u_4w'$ or $u_4u_6w'$ such that $w'\notin U\cup V$. If not, and if $w\neq w'$, then we can find a $P_4$ disjoint from $V$, a contradiction. By symmetry and linearity we can assume $w=w'$, $g=wu_2u_4\in E (G)$ and $wu_4u_6\notin G$. Then we have $h'=v_3u_3u_6$ or $v_3u_3u_7 \notin E (G)$, otherwise $v_1v_2v_3,h',u_7u_6u_5,u_5u_1w,wu_2u_4$ is a $P_5$ in $G$, a contradiction. 
    Thus $E(S)\cap E(v_3)\subseteq\{v_3u_1u_6,v_3u_2u_5,v_3u_2u_6,v_3u_2u_7\}$. By linearity, $|E(S)\cap E(v_3)|\leq 2$. Note that by the above discussion edges in $E_1^-(S)$ contains $u_4$ and thus $|E^-(S)|\leq 2$. Then we have $|E^-(S)|+|E(v_3)\cap E(S)|\le 4$, a contradiction.
    Thus we have $E_1^-(S)=\emptyset$ and $|E_2^-(S)\cup E_1^-(S)|=1$.
    
	\emphd{Case 3.} $|E'|=2$. As the discussion in case 2, $E_2^-(S)=E'$ and $E_1^-(S)=\emptyset$. Thus for any edge in $E^-(S)\cap E(v_3)$, it must contain one of $u_2,u_4$ by linearity. We conclude that $|E^-(S)|+|E(v_3)\cap E(S)|\le 4$.\\
    
	(2) $U=L_{11}$.
        
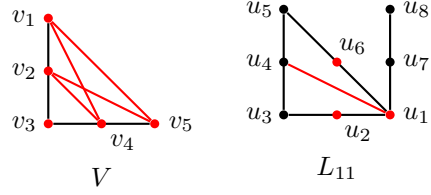
\begin{figure}[h]
	\centering
	\begin{tikzpicture}[scale=0.7, every node/.style={circle, fill, inner sep=1.2pt}]
		\node[fill=red] (v1) at (-1, 0) {};
		
		\node[fill=red] (v3) at (1, 0) {};
		
		\node[fill=red] (v5) at (-1, 2) {};

		\draw[black, thick] (v1) -- (v3);
		\draw[black, thick] (v1) -- (v5);
        \node[fill=red] (v4) at (-1, 1) {};
		\node[fill=red] (v2) at (0, 0) {};
        
		\draw[red, thick] (v3) -- (v4);
		\draw[red, thick] (v2) -- (v5);
		\draw[red, thick] (v3) -- (v5);
		\draw[red, thick] (v2) -- (v4);

		\node[draw=none,fill=none] at (-1.45, 2) {$v_1$};
		\node[draw=none,fill=none] at (-1.45, 1) {$v_2$};
		\node[draw=none,fill=none] at (-1.45, 0) {$v_3$};
		\node[draw=none,fill=none] at (0.4, -0.35) {$v_4$};
		\node[draw=none,fill=none] at (1.5, 0) {$v_5$};
		\node[draw=none,fill=none] at (0, -1) {$V$};
		
	\end{tikzpicture}%
    \hspace{0.4cm}
	\begin{tikzpicture}[scale=0.7, every node/.style={circle, fill, inner sep=1.2pt}]
	\node[fill=black] (v1) at (-1, 0) {};

	\node[fill=red] (v3) at (1, 0) {};
	\node[fill=black] (v4) at (-1, 1) {};
	\node[fill=black] (v5) at (-1, 2) {};

	\node[fill=black] (v8) at (1, 1) {};
	\node[fill=black] (v9) at (1, 2) {};

	\draw[black, thick] (v1) -- (v3);
	\draw[black, thick] (v1) -- (v5);
	\draw[black, thick] (v3) -- (v5);
	\draw[black, thick] (v3) -- (v9);
	\draw[red, thick] (v4) -- (v3);
		\node[fill=red] (v2) at (0, 0) {};
		\node[fill=red] (v6) at (0, 1) {};
		\node[draw=none,fill=none] at (-1.45, 2) {$u_5$};
	\node[draw=none,fill=none] at (-1.45, 1) {$u_4$};
	\node[draw=none,fill=none] at (-1.45, 0) {$u_3$};
	\node[draw=none,fill=none] at (0.4, -0.35) {$u_2$};
	\node[draw=none,fill=none] at (0.3, 1.3) {$u_6$};
	\node[draw=none,fill=none] at (1.5, 2) {$u_8$};
	\node[draw=none,fill=none] at (1.5, 1) {$u_7$};
	\node[draw=none,fill=none] at (1.5, 0) {$u_1$};
	\node[draw=none,fill=none] at (0, -1) {$L_{11}$};
\end{tikzpicture}	
\end{figure}

Let $S = \{u_7, u_8\}$, $T = \{v_1, v_2, v_4, v_5\}$. For any $h\in E_1(S)$, by~\Cref{4P2} and (1) we have $h\cap V \neq \emptyset$, and thus $h$ connects $U,V$. 
Since $u_7, u_8$ are not 2-centers or vertices of center pairs, by~\Cref{conpath} $h\cap V = \{v_3\}$. Let $\tilde{v} = h-\{u_7,u_8,v_3\}$. If $\tilde{v}\notin U$, then $e_1, h, u_1u_7u_8, u_1u_2u_3, u_3u_4u_5$ form a $P_5$. Thus $\tilde{v}\in U$. Observe that $\tilde{v}\neq u_1$ by linearity. If $\tilde{v}=u_2$, then $u_3u_4u_5,u_5u_6u_1,u_1u_7u_8,h,v_3v_2v_1$ form a $P_5$. So $\tilde{v}\not=u_2$ and by symmetry $\tilde{v}\not=u_6$. Hence $\tilde{v}\in \{u_3,u_4, u_5\}$. By~\Cref{contradiction} with $|S|=2$, we have $|E_1(S)|\ge 2$. So we can choose some $h\in E_{1}(S)$ such that $\tilde{v}\not=u_4$. By symmetry we can assume $h=u_8u_3v_3$.
	
By~\Cref{contradiction}, $|E(T)|\geq 6$. First, $|E_3(T)|=0$ by linearity. Next, note that $u_5u_6u_1,u_1u_2u_3,u_3u_8v_3$ and $v_1v_2v_3(v_3v_4v_5)$ already form a $P_4$. This implies that every edge $e$ in $E(T)-\{v_1v_2v_3,v_3v_4v_5\}$ connects $V$ and $U$. Thus by~\Cref{conpath}, $e$ must contain the only center pair $u_1u_4$ of $U$. So $|E_1(T)|\le 1$ and thus $|E_2(T)|\ge 5$.
	
For any edge $g\in E_2(T)-\{v_1v_2v_3,v_3v_4v_5\}$, let $\tilde{u}=g-T$.
If $u\notin U$, then $g,v_1v_2v_3,h,u_8u_7u_1,u_1u_6u_5$ form a $P_5$ in $G$, a contradiction. Thus $g$ connects $U,V$ and by~\Cref{conpath}, $\tilde{u}\in \{u_1,u_2,u_6\}$. If $\tilde{u}=u_2$, then $u_1u_6u_5,u_5u_4u_3,h,v_1v_2v_3,g$ form a $P_5$, a contradiction. So $\tilde{u}\in\{u_1,u_6\}$. By linearity, there are at most two edges in $E_2(T)$ containing $u_6$ and at most two edges in $E_2(T)$ containing $u_1$. However, $|E_2(T)|\ge 5$, so there exists at least 3 edges in $E_2(T)$ containing one of $u_1$ and $u_6$. So we can find two edges $e_1,e_2\in E_2(S)$ such that $\emptyset \neq e_1\cap e_2\subseteq S$ and $u_1\in e_1$, $u_6\in e_2$. Assume $e_3\in\{v_1v_2v_3,v_3v_4v_5\}$ does not contain $e_1\cap e_2$. Then $e_2,e_1,v_1v_7v_8,h=u_8u_3v_3,e_3$ form a $P_5$, a contradiction.\\
    
(3) $U=L_{12}$.
    
    
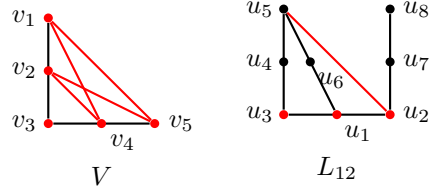
\begin{figure}[h]
	\centering
	\begin{tikzpicture}[scale=0.7, every node/.style={circle, fill, inner sep=1.2pt}]
		\node[fill=red] (v1) at (-1, 0) {};
		\node[fill=red] (v3) at (1, 0) {};
		\node[fill=red] (v5) at (-1, 2) {};
		\draw[black, thick] (v1) -- (v3);
		\draw[black, thick] (v1) -- (v5);
        \node[fill=red] (v4) at (-1, 1) {};
		\node[fill=red] (v2) at (0, 0) {};
		\draw[red, thick] (v3) -- (v4);
		\draw[red, thick] (v2) -- (v5);
		\draw[red, thick] (v3) -- (v5);
		\draw[red, thick] (v2) -- (v4);

		\node[draw=none,fill=none] at (-1.45, 2) {$v_1$};
		\node[draw=none,fill=none] at (-1.45, 1) {$v_2$};
		\node[draw=none,fill=none] at (-1.45, 0) {$v_3$};
		\node[draw=none,fill=none] at (0.4, -0.35) {$v_4$};
		\node[draw=none,fill=none] at (1.5, 0) {$v_5$};
		\node[draw=none,fill=none] at (0, -1) {$V$};
		
	\end{tikzpicture}%
    \hspace{0.4cm}
	\begin{tikzpicture}[scale=0.7, every node/.style={circle, fill, inner sep=1.2pt}]
	\node[fill=red] (v1) at (-1, 0) {};

	\node[fill=red] (v3) at (1, 0) {};
	\node[fill=black] (v4) at (-1, 1) {};
	\node[fill=black] (v5) at (-1, 2) {};
	\node[fill=black] (v6) at (1, 1) {};
	\node[fill=black] (v7) at (1, 2) {};
	\node[fill=black] (v8) at (-0.5, 1) {};

	\draw[black, thick] (v1) -- (v3);
	\draw[black, thick] (v1) -- (v5);
	\draw[black, thick] (v3) -- (v7);
	\draw[black, thick] (v5) -- (v2);
	\draw[red, thick] (v3) -- (v5);
	\node[fill=red] (v2) at (0, 0) {};
		\node[draw=none,fill=none] at (-1.45, 2) {$u_5$};
		\node[draw=none,fill=none] at (-1.45, 1) {$u_4$};
		\node[draw=none,fill=none] at (-1.45, 0) {$u_3$};
		\node[draw=none,fill=none] at (0.4, -0.35) {$u_1$};
		\node[draw=none,fill=none] at (-0.1, 0.65) {$u_6$};
		\node[draw=none,fill=none] at (1.5, 2) {$u_8$};
		\node[draw=none,fill=none] at (1.5, 1) {$u_7$};
		\node[draw=none,fill=none] at (1.5, 0) {$u_2$};
	\node[draw=none,fill=none] at (0, -1) {$L_{12}$};
\end{tikzpicture}	
\end{figure}

Let $S = \{u_7, u_8\}$, $T = \{v_1, v_2,v_4,v_5\}$.	If there is no edge $\tilde{u}_1\tilde{u}_2v_3$ such that $\tilde{u}_1\in \{u_7, u_8\}, \tilde{u}_2\in\{u_1, u_3\}$, we claim that $|E_1(S)|\le 1$. For any $h\in E_1(S)$,  by~\Cref{4P2} and case (1) we have $h\cap V \neq \emptyset$. It follows that $h\cap V = \{v_3\}$ since $u_7, u_8$ are not 2-centers or vertices of center pair. 
Let $\tilde{v} = h-\{u_7,u_8,v_3\}$, then $\tilde{v}\notin \{u_1, u_3\}$ by our assumption and $\tilde{v}\neq u_2$ by linearity. If $\tilde{v}=u_4$, then $u_5u_6u_1, u_1u_3u_2,u_2u_7u_8, h, v_3v_2v_1$ form a $P_5$. So $\tilde{v}\neq u_4$ and $\tilde{v}\neq u_6$ by symmetry. Thus $\tilde{v} = u_5$ and $|E_1(S)|\le 1$ by linearity. Thus $|E(S)|\leq 2<\frac{15}{11}|S|$, contradicting~\Cref{contradiction}.
	
Now by symmetry, we can assume $h =u_8v_3u_1\in E (G)$. Our goal compute $|E(T)|$. To start with, $|E_3(T)|=0$ by linearity. Next, note that $u_3u_4u_5,u_5u_6u_1,h=u_1u_8v_3$ and $e_1$ or $e_2$ already form a $P_4$. This implies that every edge in $E(T)-\{v_1v_2v_3,v_3v_4v_5\}$ connects $V$ and $U$, a star and a non-star subgraph. Thus by~\Cref{conpath},
for any edge in $E_1(T)$, it must the only center pair $u_2u_5$ of $U$, and thus $|E_1(T)|\le 1$. 
For any edge $g\in E_2(T)-\{v_1v_2v_3,v_3v_4v_5\}$, let $\tilde{u}=g-T$ and again by~\Cref{conpath}, $\tilde{u}$ has to be a 2-center in $U$. So $\tilde{u}\in \{u_1,u_2,u_3\}$. If $\tilde{u}=u_2$, then $u_3u_4u_5,u_5u_6u_1,h=u_8v_3u_1,v_1v_2v_3,g$ form a $P_5$, a contradiction. If $\tilde{u} = u_3$, then $u_2u_7u_8, h=u_8u_1v_3, v_3v_2v_1, g, u_3u_4u_5$ form a $P_5$, a contradiction. So $\tilde{u} = u_1$, which implies that $|E_2(T)-\{e_1,e_2\}|\le 2$ by linearity. Therefore, $|E(T)|\le 5<\frac{15}{11}|T|$, contradicting~\Cref{contradiction}.\\
    
	(4) $U=L_8$.
     \begin{figure}[h]
	\centering
	\begin{tikzpicture}[scale=0.7, every node/.style={circle, fill, inner sep=1.2pt}]
		\node[fill=red] (v1) at (-1, 0) {};
		\node[fill=red] (v3) at (1, 0) {};
		\node[fill=red] (v5) at (-1, 2) {};

		\draw[black, thick] (v1) -- (v3);
		\draw[black, thick] (v1) -- (v5);
        \node[fill=red] (v4) at (-1, 1) {};
		\node[fill=red] (v2) at (0, 0) {};
        
		\draw[red, thick] (v3) -- (v4);
		\draw[red, thick] (v2) -- (v5);
		\draw[red, thick] (v3) -- (v5);
		\draw[red, thick] (v2) -- (v4);

		\node[draw=none,fill=none] at (-1.45, 2) {$v_1$};
		\node[draw=none,fill=none] at (-1.45, 1) {$v_2$};
		\node[draw=none,fill=none] at (-1.45, 0) {$v_3$};
		\node[draw=none,fill=none] at (0.4, -0.35) {$v_4$};
		\node[draw=none,fill=none] at (1.5, 0) {$v_5$};
		\node[draw=none,fill=none] at (0, -1) {$V$};
		
	\end{tikzpicture}%
    \hspace{0.4cm}
	\begin{tikzpicture}[scale=0.7, every node/.style={circle, fill, inner sep=1.2pt}]
	\node[fill=red] (v1) at (-1, 0) {};

	\node[fill=red] (v3) at (1, 0) {};
	\node[fill=black] (v4) at (-1, 1) {};
	\node[fill=black] (v5) at (-1, 2) {};
	\node[fill=black] (v6) at (-0.5, 1) {};
	\node[fill=black] (v7) at (0, 2) {};
	\node[fill=black] (v8) at (1, 1) {};
	\node[fill=black] (v9) at (1, 2) {};

	\draw[black, thick] (v1) -- (v3);
	\draw[black, thick] (v1) -- (v5);
	\draw[black, thick] (v1) -- (v7);
	\draw[black, thick] (v3) -- (v9);
	\draw[red, thick] (v1) -- (v8);
	\draw[red, thick] (v1) -- (v9);
	\node[fill=red] (v2) at (0, 0) {};
	\node[draw=none,fill=none] at (-1.45, 2) {$u_7$};
	\node[draw=none,fill=none] at (-1.45, 1) {$u_6$};
	\node[draw=none,fill=none] at (-1.45, 0) {$u_5$};
	\node[draw=none,fill=none] at (0.4, -0.35) {$u_4$};
	\node[draw=none,fill=none] at (-0.1, 0.65) {$u_8$};
	\node[draw=none,fill=none] at (0.2, 1.5) {$u_9$};
	\node[draw=none,fill=none] at (1.5, 2) {$u_1$};
	\node[draw=none,fill=none] at (1.5, 1) {$u_2$};
	\node[draw=none,fill=none] at (1.5, 0) {$u_3$};
	\node[draw=none,fill=none] at (0, -1) {$L_8$};
\end{tikzpicture}	
\end{figure}

Let $S = \{u_6, u_7, u_8, u_9\}$. By linearity, $E_3(S)=\emptyset$. For any edge $h$ in $E_2(S)$, let $\tilde{u}=h-S$. If $\tilde{u}\in V$, then $u_1u_2u_3,u_3u_4u_5,u_5u_6u_7,h$ and an edge in $E_3(V)$ form a $P_5$, a contradiction. If $\tilde{u}\notin\{u_3,u_4,u_5\}$, then there is a $L_{11}$ consisting of $h,u_7u_6u_5,u_5u_8u_9$ and $u_5u_3u_4$, which is disjoint from $V$, contradicting (2). If $\tilde{u}\in\{u_3,u_4\}$, there is still a $L_{11}$ or $L_{12}$, consisting of $h,u_7u_6u_5,u_5u_3u_4$ and $u_1u_2u_3$, which is disjoint from $V$, contradicting case (2) or case (3). So $E_2(S)\subseteq\{u_5u_6u_7,u_5u_8u_9\}$.

For any edge $g$ in $E_1(S)$, if $g\cap V=\emptyset$, it is easy to find a $P_4,C_4,L_{11}$ or $L_{12}$ disjoint from $V$, contradicting~\Cref{4P2}, case (2) or case (3). Thus $g$ connects $V$ and $U$. However, vertices in $S$ are not 2-centers or vertices of center pairs, so by~\Cref{conpath}, $v_3\in g$. Further, $g-\{v_3\}\cup S\in\{u_3,u_4\}$, or we can easily find a $P_5$ in $G$, a contradiction. So by linearity, $|E_1(S)|\le 2$, and thus $|E(S)|\le 4<\frac{15}{11}|S|$, contradicting~\Cref{contradiction}.\\
    
	(5) $U=L_{13}$.
    \begin{figure}[h]
	\centering
	\begin{tikzpicture}[scale=0.7, every node/.style={circle, fill, inner sep=1.2pt}]
		\node[fill=red] (v1) at (-1, 0) {};
		\node[fill=red] (v3) at (1, 0) {};
		\node[fill=red] (v5) at (-1, 2) {};
		
		\draw[black, thick] (v1) -- (v3);
		\draw[black, thick] (v1) -- (v5);
        \node[fill=red] (v4) at (-1, 1) {};
		\node[fill=red] (v2) at (0, 0) {};
        
		\draw[red, thick] (v3) -- (v4);
		\draw[red, thick] (v2) -- (v5);
		\draw[red, thick] (v3) -- (v5);
		\draw[red, thick] (v2) -- (v4);

		\node[draw=none,fill=none] at (-1.45, 2) {$v_1$};
		\node[draw=none,fill=none] at (-1.45, 1) {$v_2$};
		\node[draw=none,fill=none] at (-1.45, 0) {$v_3$};
		\node[draw=none,fill=none] at (0.4, -0.35) {$v_4$};
		\node[draw=none,fill=none] at (1.5, 0) {$v_5$};
		\node[draw=none,fill=none] at (0, -1) {$V$};
		
	\end{tikzpicture}%
    \hspace{0.4cm}
	\begin{tikzpicture}[scale=0.7, every node/.style={circle, fill, inner sep=1.2pt}]
	\node[fill=red] (v1) at (-1, 0) {};
	\node[fill=red] (v3) at (1, 0) {};
	\node[fill=black] (v4) at (-1, 1) {};
	\node[fill=black] (v5) at (-1, 2) {};
	\node[fill=black] (v6) at (0, 1) {};
	\node[fill=black] (v7) at (0, 2) {};
	\node[fill=black] (v8) at (1, 1) {};
	\node[fill=black] (v9) at (1, 2) {};
	
	\draw[black, thick] (v1) -- (v3);
	\draw[black, thick] (v1) -- (v5);

	\draw[black, thick] (v3) -- (v9);
	\node[fill=red] (v2) at (0, 0) {};
	\draw[black, thick] (v2) -- (v7);
	\node[draw=none,fill=none] at (-1.45, 2) {$u_5$};
	\node[draw=none,fill=none] at (-1.45, 1) {$u_4$};
	\node[draw=none,fill=none] at (-1.45, 0) {$u_1$};
	\node[draw=none,fill=none] at (0.4, -0.35) {$u_2$};
	\node[draw=none,fill=none] at (0.4, 0.65) {$u_6$};
	\node[draw=none,fill=none] at (0.4, 1.65) {$u_7$};
	\node[draw=none,fill=none] at (1.5, 2) {$u_9$};
	\node[draw=none,fill=none] at (1.5, 1) {$u_8$};
	\node[draw=none,fill=none] at (1.5, 0) {$u_3$};
	\node[draw=none,fill=none] at (0, -1) {$L_{13}$};
\end{tikzpicture}	
\end{figure}

Let $S = \{u_4, u_5, u_6, u_7, u_8, u_9\}$.	By~\Cref{4P2}, $E_3(S)= \emptyset$. For any edge $h\in E_2(S)$, let $\tilde{u}=h-S$. If $\tilde{u}\in V$, there is a $P_5$ consisting of three edges in $U$, the edge $h$ and one edge in $V$. If $\tilde{u}\notin V$, there is a $P_4$ disjoint from $V$, contradicting~\Cref{4P2}.  
Similarly to avoid presence of $P_4$ or $C_4$ disjoint from $V$, for any edge $g\in E_1(S)$, $g$ must connect $V$ and $U$. Moreover, by~\Cref{conpath}, $v_3\in g$. Finally, the vertex $g-\{v_3\}\cup S\in\{u_1,u_2,u_3\}$, or we can find a $P_5$ in $G$, a contradiction. So by linearity, $|E_1(S)|\le 3$ and thus $|E(S)|\le 6<\frac{15}{11}|S|$, contradicting~\Cref{contradiction}.
\end{proof}
\begin{lem}\label{33}
Let $U$ and $V$ be two disjoint non-star subgraphs in $G$, then $|E(U\cup V)|\leq |E_3(U)|+|E_3(V)|+a(U)+a(V)$, where $a(X)$ is the number of center pairs in $X$, $X=U,V$.
\end{lem}
\begin{proof}
It suffices to show that for any edge $e\in E(U\cup V)-E_3(U)- E_3(V)$, $e$ contains a center pair of $U$ or $V$. 

Suppose $e$ is a connecting edge, by~\Cref{conpath}, $e$ contains  a center pair of $U$ or $V$ since $U$ or $V$ has no center.

Suppose $e\in E(U)-E_3(U)$ is not a connecting edge. If $e\in E_1(U)$, since $U$ has no center, there exists a $P_3$ contained in $U\cup\{e\}$ disjoint from $V$. However, $U$ has at least 3 edges since it is not a star, so $U\cup\{e\}$ has at least 4 edges, and thus is not an induced $P_3$, contradicting~\Cref{I3P2}. If $e\in E_2(U)$ and $e\cap U$ is not a center pair of $U$, by definition there also exists be a $P_3$ contained in $U\cup\{e\}$ disjoint from $V$, contradicting~\Cref{I3P2} for the same reason. As a result, $e$ always contains a center pair of $U$ and vice versa.
\end{proof}
\begin{coro}\label{PC33}
    There is no disjoint $P_3$ and $P_3$, or $P_3$ and $C_3$, or $C_3$ and $C_3$ in $G$.
\end{coro}
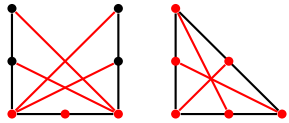
\begin{proof}
    \begin{figure}[h]
	\centering
		\begin{tikzpicture}[scale=0.7, every node/.style={circle, fill, inner sep=1.2pt}]
		\node[fill=red] (v1) at (-1, 0) {};
		\node[fill=red] (v3) at (1, 0) {};
		\node[fill=black] (v4) at (-1, 1) {};
		\node[fill=black] (v5) at (-1, 2) {};
		\node[fill=black] (v6) at (1, 1) {};
		\node[fill=black] (v7) at (1, 2) {};	
		
		\draw[black, thick] (v1) -- (v5);
		\draw[black, thick] (v3) -- (v7);
		\draw[black, thick] (v1) -- (v3);
		\node[fill=red] (v2) at (0,0) {};
		\draw[red, thick] (v1) -- (v7);
		\draw[red, thick] (v1) -- (v6);
		\draw[red, thick] (v3) -- (v4);
		\draw[red, thick] (v3) -- (v5);
	\end{tikzpicture}
	\hspace{0.4cm}
	\begin{tikzpicture}[scale=0.7, every node/.style={circle, fill, inner sep=1.2pt}]
	\node[fill=red] (v1) at (-1, 0) {};

	\node[fill=red] (v3) at (1, 0) {};

	\node[fill=red] (v5) at (-1, 2) {};

	\draw[black, thick] (v1) -- (v3);
	\draw[black, thick] (v1) -- (v5);
	\draw[black, thick] (v3) -- (v5);
	\draw[red, thick] (v3) -- (v4);
	\draw[red, thick] (v2) -- (v5);
	\node[fill=red] (v4) at (-1, 1) {};
	\node[fill=red] (v2) at (0, 0) {};
	\node[fill=red] (v6) at (0, 1) {};
	\draw[red, thick] (v1) -- (v6);
\end{tikzpicture}
\end{figure}
There are 3 edges, 4 center pairs and 7 vertices in $P_3$, and there are 3 edges, 3 center pairs and 6 vertices in $C_3$. So $|E_3(U)|+a(U)=|U|$ for $U=P_3,C_3$. Thus by~\Cref{33} and~\Cref{contradiction} there is no such disjoint subgraphs.
\end{proof}

\begin{lem}\label{3P2}
After deleting vertices of an arbitrary $P_2$ in $G$, all the components are stars.
\end{lem}

\begin{proof}
It is equivalent to show that $G-P_2$ has no $P_3$ or $C_3$. 

(1) There is no disjoint $P_2$ and $L_{15}$ in $G$.
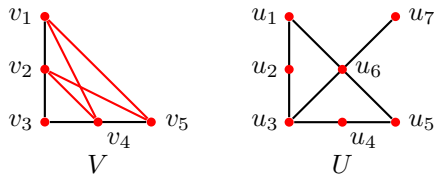
\begin{figure}[h]
	\centering
	\begin{tikzpicture}[scale=0.7, every node/.style={circle, fill, inner sep=1.2pt}]
		\node[fill=red] (v1) at (-1, 0) {};
		
		\node[fill=red] (v3) at (1, 0) {};
		
		\node[fill=red] (v5) at (-1, 2) {};

		\draw[black, thick] (v1) -- (v3);
		\draw[black, thick] (v1) -- (v5);
		\draw[red, thick] (v3) -- (v4);
		\draw[red, thick] (v2) -- (v5);
		\draw[red, thick] (v3) -- (v5);
		\draw[red, thick] (v2) -- (v4);
		\node[fill=red] (v4) at (-1, 1) {};
		\node[fill=red] (v2) at (0, 0) {};
		\node[draw=none,fill=none] at (-1.45, 2) {$v_1$};
		\node[draw=none,fill=none] at (-1.45, 1) {$v_2$};
		\node[draw=none,fill=none] at (-1.45, 0) {$v_3$};
		\node[draw=none,fill=none] at (0.4, -0.35) {$v_4$};
		\node[draw=none,fill=none] at (1.5, 0) {$v_5$};
		\node[draw=none,fill=none] at (0, -0.8) {$V$};
	\end{tikzpicture}
	\hspace{0.4cm}
	\begin{tikzpicture}[scale=0.7, every node/.style={circle, fill, inner sep=1.2pt}]
		\node[fill=red] (v1) at (-1, 0) {};
		
		\node[fill=red] (v3) at (1, 0) {};
		
		\node[fill=red] (v5) at (-1, 2) {};
		
		\node[fill=red] (v7) at (1, 2) {};
		\draw[black, thick] (v1) -- (v7);
		\draw[black, thick] (v1) -- (v3);
		\draw[black, thick] (v1) -- (v5);
		\draw[black, thick] (v3) -- (v5);

		\node[fill=red] (v4) at (-1, 1) {};
		\node[fill=red] (v2) at (0, 0) {};
		\node[fill=red] (v6) at (0, 1) {};

		\node[draw=none,fill=none] at (-1.45, 2) {$u_1$};
		\node[draw=none,fill=none] at (-1.45, 1) {$u_2$};
		\node[draw=none,fill=none] at (-1.45, 0) {$u_3$};
		\node[draw=none,fill=none] at (0.4, -0.35) {$u_4$};
		\node[draw=none,fill=none] at (0.5, 1) {$u_6$};
		\node[draw=none,fill=none] at (1.5, 0) {$u_5$};
		\node[draw=none,fill=none] at (1.5, 2) {$u_7$};
		\node[draw=none,fill=none] at (0, -0.8) {$U$};
	\end{tikzpicture}

\end{figure}

If not, there is a 2-path $V$ consisting of $e_1=v_1v_2v_3, e_2 = v_3v_4v_5$ and disjoint $U = L_{15}$ consisting of $f_1=u_1u_2u_3,$ $ f_2 = u_3u_4u_5, f_3=u_5u_6u_1, f_4=u_3u_6u_7$. Let $S=U\cup V$. 

Assume there is a vertex $w$ outside $S$. By~\Cref{lmcon} $G-v_3$ is connected, we could suppose there is an edge $g$ that connecting $S-v_3$ and $w$ in $G-v_3$ (by taking appropriate $w$). By~\Cref{conpath}, if $g$ intersect both $U$ and $V$, it must contain a 2-center in $V-v_3$ as well as a center pair in $U$, contradicting the fact that $w\in g$. Then, if $g$ only intersects $U$, since there is no center pair or centers in $U$, there is always a $P_3$ in $U\cup g$ which is not an induced $P_3$, contradicting~\Cref{I3P2}. Finally, if $g$ only intersects $V$, it will form a $P_3$ or $C_3$ with $V$, contradicting~\Cref{PC33}. Thus, there is no vertex outside $S$.

Now it remains to count $E(S)= E_3(V)\cup E_3(U)\cup E_1(V) \cup E_1(U)$. First it is clear that $|E_3(V)|=2,$ $ |E_3(U)|\leq 7$ by linearity. Next by~\Cref{conpath} $|E_1(U)|$ is bounded by the number of center pairs in $V$ and $|E_1(V-v_3)|$ vice versa. Thus $E_1(U)\le 4$ and $|E_1(V-v_3)|=0$. Finally $|E_1(v_3)|\le \lfloor \frac{7}{2}\rfloor = 3$ by linearity. Thus $|E(S)|\leq 16<\frac{15}{11}|S|$, contradicting~\Cref{contradiction}.\\

(2) There is no disjoint $P_2$ and $P_3$ in $G$.
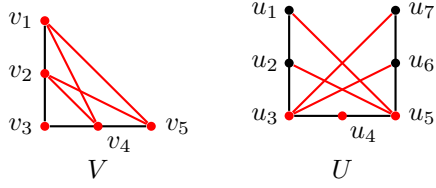
\begin{figure}[h]
	\centering
	\begin{tikzpicture}[scale=0.7, every node/.style={circle, fill, inner sep=1.2pt}]
		\node[fill=red] (v1) at (-1, 0) {};
		
		\node[fill=red] (v3) at (1, 0) {};
		
		\node[fill=red] (v5) at (-1, 2) {};

		\draw[black, thick] (v1) -- (v3);
		\draw[black, thick] (v1) -- (v5);
		\draw[red, thick] (v3) -- (v4);
		\draw[red, thick] (v2) -- (v5);
		\draw[red, thick] (v3) -- (v5);
		\draw[red, thick] (v2) -- (v4);
		\node[fill=red] (v4) at (-1, 1) {};
		\node[fill=red] (v2) at (0, 0) {};
		\node[draw=none,fill=none] at (-1.45, 2) {$v_1$};
		\node[draw=none,fill=none] at (-1.45, 1) {$v_2$};
		\node[draw=none,fill=none] at (-1.45, 0) {$v_3$};
		\node[draw=none,fill=none] at (0.4, -0.35) {$v_4$};
		\node[draw=none,fill=none] at (1.5, 0) {$v_5$};
		\node[draw=none,fill=none] at (0, -0.8) {$V$};
	\end{tikzpicture}
	\hspace{0.4cm}
	\begin{tikzpicture}[scale=0.7, every node/.style={circle, fill, inner sep=1.2pt}]
	\node[fill=red] (v1) at (-1, 0) {};
	\node[fill=red] (v3) at (1, 0) {};
	\node[fill=black] (v4) at (-1, 1) {};
	\node[fill=black] (v5) at (-1, 2) {};
	\node[fill=black] (v6) at (1, 1) {};
	\node[fill=black] (v7) at (1, 2) {};	
	\node[draw=none,fill=none] at (-1.45, 2) {$u_1$};
	\node[draw=none,fill=none] at (-1.45, 1) {$u_2$};
	\node[draw=none,fill=none] at (-1.45, 0) {$u_3$};
	\node[draw=none,fill=none] at (0.4, -0.35) {$u_4$};
	\node[draw=none,fill=none] at (1.5, 2) {$u_7$};
	\node[draw=none,fill=none] at (1.5, 1) {$u_6$};
	\node[draw=none,fill=none] at (1.5, 0) {$u_5$};
	
	\draw[black, thick] (v1) -- (v5);
	\draw[black, thick] (v3) -- (v7);
	\draw[black, thick] (v1) -- (v3);
	\node[fill=red] (v2) at (0,0) {};
	\draw[red, thick] (v1) -- (v7);
	\draw[red, thick] (v1) -- (v6);
	\draw[red, thick] (v3) -- (v4);
	\draw[red, thick] (v3) -- (v5);
	\node[draw=none,fill=none] at (0, -1) {$U$};
\end{tikzpicture}
\end{figure}

If not, then there are a 2-path $V$ consisting of $e_1=v_1v_2v_3, e_2 = v_3v_4v_5$ and another disjoint 3-path $U$ consisting of $f_1=u_1u_2u_3, f_2 = u_3u_4u_5, f_3=u_5u_6u_7$. Let $S=U\cup V$. 

Assume there is a vertex $w$ in $G-S$. By~\Cref{lmcon} $G-v_3$ is connected, thus we could suppose there is an edge $g$ connecting $S-v_3$ and $w$. By~\Cref{I3P2} $g$ must intersect $V$ and by~\Cref{PC33} $g$ must intersect $U$. Thus it connects $V$ and $U$. However, this again contradicting~\Cref{conpath} that $g$ must contain a 2-center in $V-v_3$ as well as a center pair in $U$.

Now it remains to count $E(S) = E_3(V)\cup E_3(U)\cup E_1(V) \cup E_2(V)$. First it is clear that $|E_3(V)|=2$ by linearity and $|E_3(U)|=3$ since $U$ is an induced $P_3$ by~\Cref{I3P2}. 
Next by~\Cref{conpath} $|E_2(V)|\le a(V)=4$ and $|E_1(V-v_3)|\le a(U)=4$. Finally $|E_1(v_3)|\le \lfloor \frac{7}{2}\rfloor = 3$ by linearity. Thus$|E(S)|\leq 16<\frac{15}{11}|S|$, contradicting~\Cref{contradiction}.

(3) There is no disjoint $P_2$ and $C_3$ in $G$.
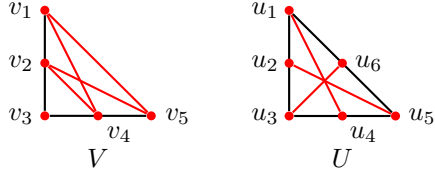
\begin{figure}[h]
	\centering
	\begin{tikzpicture}[scale=0.7, every node/.style={circle, fill, inner sep=1.2pt}]
		\node[fill=red] (v1) at (-1, 0) {};
		
		\node[fill=red] (v3) at (1, 0) {};
		
		\node[fill=red] (v5) at (-1, 2) {};

		\draw[black, thick] (v1) -- (v3);
		\draw[black, thick] (v1) -- (v5);
		\draw[red, thick] (v3) -- (v4);
		\draw[red, thick] (v2) -- (v5);
		\draw[red, thick] (v3) -- (v5);
		\draw[red, thick] (v2) -- (v4);
		\node[fill=red] (v4) at (-1, 1) {};
		\node[fill=red] (v2) at (0, 0) {};
		\node[draw=none,fill=none] at (-1.45, 2) {$v_1$};
		\node[draw=none,fill=none] at (-1.45, 1) {$v_2$};
		\node[draw=none,fill=none] at (-1.45, 0) {$v_3$};
		\node[draw=none,fill=none] at (0.4, -0.35) {$v_4$};
		\node[draw=none,fill=none] at (1.5, 0) {$v_5$};
		\node[draw=none,fill=none] at (0, -0.8) {$V$};
	\end{tikzpicture}
	\hspace{0.4cm}
	\begin{tikzpicture}[scale=0.7, every node/.style={circle, fill, inner sep=1.2pt}]
		\node[fill=red] (v1) at (-1, 0) {};
		
		\node[fill=red] (v3) at (1, 0) {};
		
		\node[fill=red] (v5) at (-1, 2) {};

		\draw[black, thick] (v1) -- (v3);
		\draw[black, thick] (v1) -- (v5);
		\draw[black, thick] (v3) -- (v5);
		\draw[red, thick] (v3) -- (v4);
		\draw[red, thick] (v2) -- (v5);
		\node[fill=red] (v4) at (-1, 1) {};
		\node[fill=red] (v2) at (0, 0) {};
		\node[fill=red] (v6) at (0, 1) {};
		\draw[red, thick] (v1) -- (v6);
		\node[draw=none,fill=none] at (-1.45, 2) {$u_1$};
		\node[draw=none,fill=none] at (-1.45, 1) {$u_2$};
		\node[draw=none,fill=none] at (-1.45, 0) {$u_3$};
		\node[draw=none,fill=none] at (0.4, -0.35) {$u_4$};
		\node[draw=none,fill=none] at (0.5, 1) {$u_6$};
		\node[draw=none,fill=none] at (1.5, 0) {$u_5$};
		\node[draw=none,fill=none] at (0, -0.8) {$U$};
	\end{tikzpicture}
\end{figure}

If not, then there is a 2-path $V$ consisting of $e_1=v_1v_2v_3, e_2 = v_3v_4v_5$ and another disjoint 3-cycle $U$ consisting of $f_1=u_1u_2u_3, f_2 = u_3u_4u_5, f_3=u_5u_6u_1$. Let $S=U\cup V$. 

Exactly the same as in (1) and (2), it follows that there are no vertices in $G-S$. Thus, $E(S) = E_3(V)\cup E_3(U)\cup E_1(V) \cup E_1(U)$, where $E_1(V)=E_2(U)$, $E_1(U)=E_2(V)$. It suffices to count $E(S)$. 

To start with, assume there is an edge $g = u_2u_4u_6$, then $|E_3(V)|=2$, $|E_3(U)|= 4$, $|E_1(U)|=|E_2(V)|\le 4$ and $|E_1(V)|=|E_2(U)|\le 3$ by linearity. Thus$|E(S)|\leq 13<\frac{15}{11}|S|$, contradicting~\Cref{contradiction}. 

If there is no such $g = u_2u_4u_6$, then $|E_3(V)|=2$, $|E_3(U)|= 3$. We still have $|E_1(U)|=|E_2(V)|\le 4$ by linearity. For $E_1(V)$, first consider $E_1(V-v_3)$ must intersect the center pair in $U$, thus $|E_1(V-v_3)|\le 3$. Then $E_1(v_3)$ can contribute to the increase of $|E_1(V)|$ by using pairs between $u_2,u_4$ and $u_6$, and by linearity there would be at most one such pair. So $|E_1(V)|\le 3+1=4$. Thus $|E(S)|\leq 13<\frac{15}{11}|S|$, contradicting~\Cref{contradiction}. 
\end{proof}

\begin{lem}\label{2P2}
    For $U = S_3$ or $P_2$, there is no disjoint $U$ and $P_2$ in $G$.
\end{lem}

\begin{proof}
(1) For the case $U=S_3$:

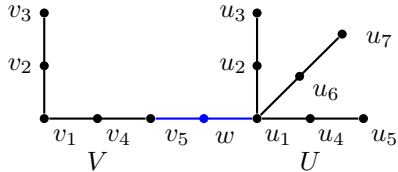
\begin{figure}[h]
	\centering
	\begin{tikzpicture}[scale=0.7, every node/.style={circle, fill, inner sep=1.2pt}]
		\node[fill=black] (v1) at (-1, 0) {};
		
		\node[fill=black] (v3) at (1, 0) {};
		
		\node[fill=black] (v5) at (-1, 2) {};

		\draw[black, thick] (v1) -- (v3);
		\draw[black, thick] (v1) -- (v5);
		\node[fill=black] (v4) at (-1, 1) {};
		\node[fill=black] (v2) at (0, 0) {};
		\node[draw=none,fill=none] at (-1.45, 2) {$v_3$};
		\node[draw=none,fill=none] at (-1.45, 1) {$v_2$};
		\node[draw=none,fill=none] at (-0.6, -0.35) {$v_1$};
		\node[draw=none,fill=none] at (0.4, -0.35) {$v_4$};
		\node[draw=none,fill=none] at (1.5, -0.35) {$v_5$};
		\node[draw=none,fill=none] at (0, -0.8) {$V$};

		\node[fill=black] (u1) at (3, 0) {};
		\draw[blue,thick](v3) -- (u1);
		\node[fill=blue] (w) at (2, 0) {};
		\node[draw=none,fill=none] at (2.4, -0.35) {$w$};
		\node[fill=black] (u3) at (5, 0) {};
		
		\node[fill=black] (u5) at (3, 2) {};

		\draw[black, thick] (u1) -- (u3);
		\draw[black, thick] (u1) -- (u5);

		\node[fill=black] (u4) at (3, 1) {};
		\node[fill=black] (u2) at (4, 0) {};
		\node[fill=black] (u7) at (4.6, 1.6) {};
		\draw[black,thick](u1) -- (u7);
		\node[fill=black] (u6) at (3.8, 0.8) {};
		\node[draw=none,fill=none] at (2.55, 2) {$u_3$};
		\node[draw=none,fill=none] at (2.55, 1) {$u_2$};
		\node[draw=none,fill=none] at (3.4, -0.35) {$u_1$};
		\node[draw=none,fill=none] at (4.4, -0.35) {$u_4$};
		\node[draw=none,fill=none] at (5.3, 1.5) {$u_7$};
		\node[draw=none,fill=none] at (4.3, 0.5) {$u_6$};
		\node[draw=none,fill=none] at (5.4, -0.35) {$u_5$};
		\node[draw=none,fill=none] at (4, -0.8) {$U$};
		
	\end{tikzpicture}
\end{figure}
If not, there are a 2-star $V$ consisting of $e_1=v_1v_2v_3, e_2 = v_1v_4v_5$ and another disjoint 3-star $U$ consisting of $f_1=u_1u_2u_3, f_2 = u_1u_4u_5, f_3=u_1u_6u_7$. Define $F_V = \{e_1,e_2\}, F_U=\{f_1,f_2,f_3\}$, $S=N_G(v_1)\cup N_G(u_1)$. And we define $W$ to be all the possible vertices outside $U\cup V$ that are incident to $u_1$ as well as one of $v_2,v_3,v_4$ or $v_5$. By linearity $|W|\le 4$. We claim that every edge in $E(S-W)$ is incident to one of $v_1$ or $u_1$. As a result, every vertex in $S-W$ is of degree 2 and thus $|E(S-W)|\leq |S-W|<\frac{15}{11}|S-W|$, contradicting~\Cref{contradiction}. Now it suffices to prove the claim. 

Note that for any vertex $x\in S-W$, by using the edges in $F_V$ and $F_U$ together with the extra edge incident $x$, we can always find disjoint 3-star and 2-star such that $x$ is a leaf of one of the stars, and meanwhile $v_1$ and $u_1$ are the centers of the stars. Thus to prove the claim, it suffices to prove every edge incident to the leaves of $V$ or $U$ is incident to one of $v_1$ or $u_1$. That is, $E(V)\cup E(U)\subseteq E(v_1)\cup E(u_1)$.

First we show $E(U)\subseteq E(u_1)\cup E(v_1)$. Note that by~\Cref{3P2} $E_3(U-u_1)=\emptyset$. Again by~\Cref{3P2}, any edge $f\in E_2(U-u_1)-F_U$ must be incident to $V$. Thus $f$ is incident to $v_1$, otherwise there would be a $P_5$ in $G$ consisting of 2 edges in $V$, 2 edges in $U$ and $g$. Still by~\Cref{3P2}, any edge $g\in E_1(U-u_1)$ intersects $V$. If $g\in E_2(V)$, without loss of generosity we can assume $g=\{u_3,v_3,v_5\}$, and there would be disjoint $C_3$ consisting of $g,e_1,e_2$ and $P_2$ consisting of $f_2,f_3$, contradicting~\Cref{3P2}. So $g\in E_1(V)$ and $f \cap V = v_1$, otherwise there would be a $P_5$ in $G$. Thus $E(U)\subseteq E(v_1)\cup E(u_1)$. 

Similarly, $E_3(V-v_1)=\emptyset$ by linearity. Any edge in $E_2(V-v_1)$ or $E_1(V-v_1)$ must intersect $u_1$, otherwise there would be a pair of disjoint 3-cycle and 2-star or a $P_5$ in $G$, contradiction. Thus $E(V)\subseteq E(v_1)\cup E(u_1)$.

(2)For the case $U=P_2$:

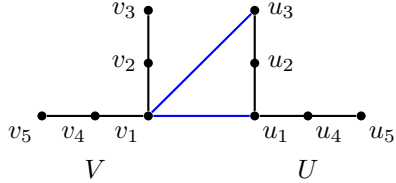
\begin{figure}[h]
	\centering
	\begin{tikzpicture}[scale=0.7, every node/.style={circle, fill, inner sep=1.2pt}]
		\node[fill=black] (v1) at (-1, 0) {};
		
		\node[fill=black] (v3) at (1, 0) {};
		
		\node[fill=black] (v5) at (1, 2) {};

		\draw[black, thick] (v1) -- (v3);
		\draw[black, thick] (v3) -- (v5);
		\node[fill=black] (v4) at (1, 1) {};
		\node[fill=black] (v2) at (0, 0) {};
		\node[draw=none,fill=none] at (0.55, 2) {$v_3$};
		\node[draw=none,fill=none] at (0.55, 1) {$v_2$};
		\node[draw=none,fill=none] at (-1.4, -0.35) {$v_5$};
		\node[draw=none,fill=none] at (-0.4, -0.35) {$v_4$};
		\node[draw=none,fill=none] at (0.6, -0.35) {$v_1$};
		\node[draw=none,fill=none] at (0, -1) {$V$};
		
		\node[fill=black](u1) at (3, 0) {};
		\node[fill=black](u3) at (5, 0) {};
		\node[fill=black](u5) at (3, 2) {};
		\draw[black, thick] (u1) -- (u3);
		\draw[black, thick] (u1) -- (u5);
		\node[fill=black] (u4) at (3, 1) {};
		\node[fill=black](u2) at (4, 0) {};
		
		\draw[blue, thick] (v3) -- (u1);
		\draw[blue, thick] (v3) -- (u5);
		
		\node[draw=none,fill=none] at (3.5, 2) {$u_3$};
		\node[draw=none,fill=none] at (3.5, 1) {$u_2$};
		\node[draw=none,fill=none] at (3.4, -0.35) {$u_1$};
		\node[draw=none,fill=none] at (4.4, -0.35) {$u_4$};
		\node[draw=none,fill=none] at (5.4, -0.35) {$u_5$};
		\node[draw=none,fill=none] at (4, -1) {$U$};

	\end{tikzpicture}
		\caption{Two cases of connecting outside vertices}
		\label{F2P2}
\end{figure}

If not, there are a 2-star $V$ consisting of $e_1=v_1v_2v_3, e_2 = v_1v_4v_5$ and another disjoint 2-star $U$ consisting of $f_1=u_1u_2u_3, f_2 = u_1u_4u_5$. Let $S = U\cup V$.
Firstly, $E_1(S) = \emptyset$ by (1) and~\Cref{3P2}. Next, by linearity $E_3(S)\leq 12$. So there are at least two edges in $E_2(S)$, otherwise $|E(S)|\leq 13<\frac{15}{11}|S|$, contradicting~\Cref{contradiction}.
For any edge $f\in E_2(S)$ and $w = f -S$, note that $f$ must be incident to at least one of $u_1$ and $v_1$ by~\Cref{3P2}, and thus we have $f = \{w, v_1, u_1\}$, where $i=1$ or $j=1$. Let $W=\{w\in V(G)-S\mid \exists\; f\in E_2(S),\text{ s.t. }w=f-S \text{ and } v_1wu_1\notin E(G)\}$.

\begin{clm}\label{Exonly}
	$|W|\leq 1$.
\end{clm}

\begin{figure}[h]
	\centering
	\begin{tikzpicture}[scale=0.7, every node/.style={circle, fill, inner sep=1.2pt}]
		\node[fill=black] (v1) at (-1, 0) {};
		
		\node[fill=black] (v3) at (1, 0) {};
		
		\node[fill=black] (v5) at (1, 2) {};

		\draw[black, thick] (v1) -- (v3);
		\draw[black, thick] (v3) -- (v5);
		\node[fill=black] (v4) at (1, 1) {};
		\node[fill=black] (v2) at (0, 0) {};
		\node[draw=none,fill=none] at (0.55, 2) {$v_3$};
		\node[draw=none,fill=none] at (0.55, 1) {$v_2$};
		\node[draw=none,fill=none] at (-1.4, -0.35) {$v_5$};
		\node[draw=none,fill=none] at (-0.4, -0.35) {$v_4$};
		\node[draw=none,fill=none] at (0.6, -0.35) {$v_1$};
		\node[draw=none,fill=none] at (0, -1) {$V$};
		
		\node[fill=black](u1) at (3, 0) {};
		\node[fill=black](u3) at (5, 0) {};
		\node[fill=black](u5) at (3, 2) {};
		\draw[black, thick] (u1) -- (u3);
		\draw[black, thick] (u1) -- (u5);
		\node[fill=black] (u4) at (3, 1) {};
		\node[fill=black](u2) at (4, 0) {};
		
		\draw[blue, thick] (v5) -- (u1);
		\draw[blue, thick] (v3) -- (u5);
		\node[fill=blue] (w1) at (1.6, 0.6) {};
		\node[fill=blue](w2) at (2.4, 0.6) {};
		\node[draw=none,fill=none] at (2.4, 0.05) {$w_2$};
		\node[draw=none,fill=none] at (1.6, 0.05) {$w_1$};

		\node[draw=none,fill=none] at (3.5, 2) {$u_3$};
		\node[draw=none,fill=none] at (3.5, 1) {$u_2$};
		\node[draw=none,fill=none] at (3.4, -0.35) {$u_1$};
		\node[draw=none,fill=none] at (4.4, -0.35) {$u_4$};
		\node[draw=none,fill=none] at (5.4, -0.35) {$u_5$};
		\node[draw=none,fill=none] at (4, -1) {$U$};
		
	\end{tikzpicture}
\end{figure}

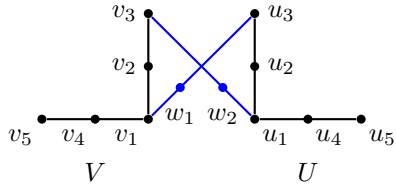
\textit{Proof of~\Cref{Exonly}}: We first suppose there exist two different vertices $w_1,w_2\in W$, such that $v_1w_1u_3,$ $u_1w_2v_2\in E(G)$ and compute $|E(S\cup W)|$ to obtain a contradiction.
For any $f\in E(w_2)- E(S)$, the five edges $w_2v_3u_1,u_1u_2u_3,$ $u_3w_1v_1,v_1v_4v_5$ and $f$ form a $P_5$ in $G$, a contradiction, and thus $E(w_2)\subseteq E(S)$.
Since $E_1(S)=\emptyset$, we have $E(w_2)\subseteq E_2(S)$. By symmetry, $E(W)\subseteq E_2(S)$. By the definition of $W$, it follows that $E(W)=E_2(S)-E_2(\{v_1,u_1\})$. Since each vertex in $G$ has degree at least 2, by linearity each vertex in $W$ has degree exactly 2, and is adjacent to both of $v_1,u_1$.
Thus $|E_2(S)|\leq 1+|E(W)|\leq 2|W|+1$. Further by linearity, $|W|\le 4$.

Consider a new pair of disjoint $P_2$, one consisting of $v_1v_4v_5$ and $v_1w_1u_3$ and the other consisting of $u_1u_4u_5$ and $u_1w_2v_3$. Now $v_2$ and $u_2$ plays the role of $w_1,w_2$ for this pair of $P_2$. Similarly to the previous discussion, 
the two vertices $v_2,u_2$ have degree 2.
So edges in $E_3(S)$ consists of $e_2$, $f_2$, at most four edges incident to $v_2,u_2$ and the edges in $F\triangleq E_3(S-\{v_2,u_2\})-F_U-F_V$. By linearity, it is easy to see that $|F|\leq 4$. In conclusion, we have
$$|E(S\cup W)|=|E_2(S)|+|E_3(S)|\leq 2|W|+1+2+4+4< \frac{15}{11}|S\cup W|,\; |W|\le 4,$$
contradicting~\Cref{contradiction} and thus there do not exist two different vertices $w_1,w_2\in W$, such that $v_1w_1u_3,$ $u_1w_2v_2\in E(G)$

If $|W|\geq 2$, then by the previous discussion, without loss of generosity we can assume $W$ is not adjacent to $u_1$.
Then $E_2(S-v)=E_1(S-v)=\emptyset$ and $W$ is not connected to $S-v_1$ in $G-v_1$, contradicting the 2-connectivity of $G$~\Cref{lmcon}. That is, $|W|\leq 1$.
\hfill\qedsymbol

Now we continue on the proof of~\Cref{2P2}.
Let $W'=W\cup \{x\in V(G)\;:\; v_1xu_1\in E(G)\}$. Then $|W'|\leq |W|+1\leq 2$. If $V(G)=S\cup W'$, then $E_2(S)=E(W')=E(W')\cap E(\{v_1,u_1\})$ and by linearity $d_G(w)\leq 2$ for any $w\in W'$. Thus $E(G)\leq E_3(S)+E_2(S)\leq 8+2|W'|<\frac{15}{11}|S\cup W'|$. Thus $S\cup W'\subsetneq V(G)$. Since $G-W'$ is disconnected, by~\Cref{lmcon} $|W'|=2$. We may assume $W=\{w\}$ and $W'=\{w,x\}$. Since $w\in W$, by symmetry we can assume $v_1wu_3\in E(G)$.
If $E(w)-E(S)\neq \emptyset$, assume $wy_1y_2\in E(w)-E(S)$. Since $\delta(G)\geq 2$ and $y_1\notin W$, there exists an edge $y_1z_1z_2\in E(G)-E(S\cup \{w,y_1,y_2\})$. Then $f_2,f_1,u_3wv_1,wy_1y_2,y_1z_1z_2$ form a $P_5$ in $G$, a contradiction. Thus we can let $E(w)=\{u_3wv_1, v_3wu_1\}$ and thus $w,z$ are not adjacent in $G$. Now note that $z$ is disconnect to $S$ in $G-v_1$, contradicting~\Cref{lmcon}.
\end{proof}

\section{Proof of~\Cref{main}}\label{Section4}

\begin{defi}
    Given a linear 3-graph $G$ and a vertex set $S$ of $G$. The edge-labeled link graph $(H,l)$ of $S$, denoted by $(H,l)=L(G,S)$, is the simple graph $H$ with an edge labeling $l$, such that $V(H)=V(G)-S,  E (H)=\{xy: x,y\in V(H) \mbox{ and }(x,y,z)\in E (G) \mbox{ for some } z\in S \}$, and $l(xy)=z$ if $\{x, y, z\} \in E (G)$ for some  $ z \in S $. Note that $l$ is well-defined by the linearity of $G$.
\end{defi}

\begin{lem}\label{lm8}
	For any vertex $v$ in $G$, if $d_G(v)\ge 6$, then $d_G(u)\le 3$ for any neighbor $u$ of $v$. If $d_G(v)\ge 5$, then $d_G(u)\le 2$ for any non-neighbor $u$ of $v$.
\end{lem}

\begin{proof}
	Suppose not, in both cases we consider the edge-labeled link graph $G(v,u)$, denoted by $H$. We can always find five edges labeled $v$ together with three edges labeled by $u$. Then any two edges labeled $u$ must intersect four edges labeled $v$. It implies that each edge labeled by $u$ must intersect two edges labeled by $v$. However, there are only five edges labeled $v$. By the pigeonhole principle, there must be one edge labeled by $v$ intersect two edges labeled by $u$, say $e_1$ and $e_2$ labeled by $u$. Now $e_1$ and $e_2$ can intersect at most 3 edges labeled $v$, a contradiction.
\end{proof}

Now we give a bound on the degree list of a fixed edge of $G$.
For the minimal counterexample $G$, fix any edge $e=\{a,b,c\}\in E (G)$. Denote the edge-labelled link graph $(G(e),l)$ by $(H,l)$. Note that by linearity, every edge in $H$ has exactly one label, and every class of labelled edges forms a matching in $H$. Without loss of generality, we can assume $d_G(a)\geq d_G(b)\geq d_G(c)$ and define degree list of $e$ as $d_G(e) = ((d_G(a),d_G(b),d_G(c))$. We say a triple of numbers $(x_1,x_2,x_3)\geq (y_1,y_2,y_3)$, if $x_i\geq y_i,i=1,2,3$.

\begin{lem}\label{degreelist}
    There is no edge $e = \{a,b,c\}\in E (G)$ such that $d_G(e)\geq (5,5,4)$.
\end{lem}
\begin{proof}
Suppose $(d_G(a),d_G(b),d_G(c))\geq (5,5,4)$, then we could take 4 edges labeled by $a$, 4 edges labeled by $b$ and 3 edges labeled by $c$ in the link graph $G(e)$, denoted by $H$. In the following proof, we color edge of label $a$, $b$ and $c$ by red, green and blue respectively. Note that each color is a matching.

Now consider the subgraph $H_1$ of red and green edges, $H_2$ of red and blue edges, $H_3$ of blue and green edges (delete isolated vertices). Each of them is a union of two matchings, and thus disjoint union of paths and even cycles.

%\textcolor{red}{change proof of * to begin proof}

\begin{clm}\label{H123}
   The red-green graph $H_1$ is a $C_8$ and every blue edge is a chord in the red-green cycle.
\end{clm}
\textit{Proof of~\Cref{H123}}: 
To start with, any two monochromatic edges in $H$ correspond to a path of length two in $G$, so there are no two disjoint pairs of monochromatic edges with different colors in $H$ by~\Cref{2P2}. 

Suppose there is a $C_4$ in $H_i$ for some $i=1,2,3$, colored alternatively by $c_1$ and $c_2$. Apparently the $C_4$ is a component of $H_i$. Without loss of generality, we could suppose there are four edges colored in $c_1$. Then there are two edges colored in $c_1$ disjoint from this $C_4$. They are monochromatic and disjoint from two other monochromatic edges colored in $c_2$ in the cycle, a contradiction. To conclude, there is no $C_4$ in $H_i,\forall i\in[3]$.

Suppose there is an isolated edge $e_1$ colored in $c_1$ in $H_i$ for some $i=1,2,3$, and let the other color in $H_i$ be $c_2$. We first claim that $c_2$ must be blue. Otherwise there are four edges colored in $c_2$, we choose arbitrarily another edge $e_2$ colored in $c_1$. Then $e_1$ and $e_2$ intersect at most two edges colored in $c_2$. So there exist two edges colored $c_2$ disjoint from $e_1$ and $e_2$,  which is a pair of monochromatic edges colored in $c_2$, a contradiction. Thus $e_1$ must be red or green and by symmetry we can assume $e_1$ is red and $i=2$.
By the above discussion the red-green graph $H_1$ has no isolated edge, so there is a green edge $e_2$ intersecting $e_1$. Since $e_1$ is disjoint from blue edges, the red-green $P_2$ consisting of $e_1$ and $e_2$ can intersect at most one blue edge. There are three blue edges in total, so there are two blue edges disjoint from this red-green $P_2$. It follows that there are two disjoint $P_2$ in $G$,contradicting~\Cref{2P2}. Therefore, there is no isolated red edge in $H_2$. Similarly, there is no isolated green edge in $H_3$. To conclude, there is no isolated edge in $H_i,\forall i \in [3]$.

Next, let $e_1$ be an edge colored in $c_1$. There always exists another color $c_2$ with four edges colored in $c_2$. We claim that $e_1$ must intersect two edges of $c_2$. Since there is no isolated edge in $H_i$ for any $i=1,2,3$, there is an edge $e_2$ colored in the third color $c_3$ intersecting $e_1$. By~\Cref{2P2}, there are no two edges colored in $c_2$ disjoint from $e_1$ and $e_2$. So $e_1$ and $e_2$ must intersect three edges colored in $c_2$, which implies that both $e_1$ and $e_2$ intersect two edges colored in $c_2$. To conclude, every red edge intersect two green edges, every green edge intersect two red edges, and every blue edge intersect both two red edges and two green edges (See \Cref{local}).

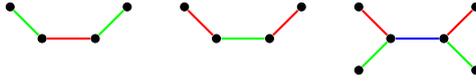
\begin{figure}[h]
	\centering
    \begin{tikzpicture}[scale=0.7, every node/.style={circle, fill, inner sep=1.2pt}]
	\node[fill=black] (v1)at(-1,0) {};
	\node[fill=black] (v2) at (0, 0) {};
	\node[fill=black] (v3) at (0.6, 0.6) {};
	\node[fill = none](v4) at (0.6, -0.6) {};
    \node[fill=black] (v5) at (-1.6, 0.6) {};
	\node[fill = none] (v6) at (-1.6, -0.6) {};
	\draw[red, thick] (v1) -- (v2);
	\draw[green, thick] (v2) -- (v3);
	\draw[green, thick] (v1) -- (v5);
\end{tikzpicture}
\hspace{0.4cm}
\begin{tikzpicture}[scale=0.7, every node/.style={circle, fill, inner sep=1.2pt}]
	\node[fill=black] (v1)at(-1,0) {};
	\node[fill=black] (v2) at (0, 0) {};
	\node[fill=black] (v3) at (0.6, 0.6) {};
	\node[fill = none](v4) at (0.6, -0.6) {};
    \node[fill=black] (v5) at (-1.6, 0.6) {};
	\node[fill = none] (v6) at (-1.6, -0.6) {};
	\draw[green, thick] (v1) -- (v2);
	\draw[red, thick] (v2) -- (v3);
	\draw[red, thick] (v1) -- (v5);
\end{tikzpicture}
\hspace{0.4cm}
\begin{tikzpicture}[scale=0.7, every node/.style={circle, fill, inner sep=1.2pt}]
	\node[fill=black] (v1)at(-1,0) {};
	\node[fill=black] (v2) at (0, 0) {};
	\node[fill=black] (v3) at (0.6, 0.6) {};
	\node[fill=black] (v4) at (0.6, -0.6) {};
    \node[fill=black] (v5) at (-1.6, 0.6) {};
	\node[fill=black] (v6) at (-1.6, -0.6) {};
	\draw[blue, thick] (v1) -- (v2);
	\draw[red, thick] (v2) -- (v3);
    \draw[green, thick] (v2) -- (v4);
	\draw[red, thick] (v1) -- (v5);
	\draw[green, thick] (v1) -- (v6);
\end{tikzpicture}
\caption{Local structures of red, green and blue edge}
\label{local}
\end{figure}

By the above discussion, $H_1$ is a 2-regular graph  with 8 edges and there is no $C_4$ in $H_1$. Thus $H_1$ is a $C_8$. Moreover, the endpoints of any blue edges are vertices of degree two in $H_1$, so every blue edge is a chord in the red-green cycle. \hfill\qedsymbol\\

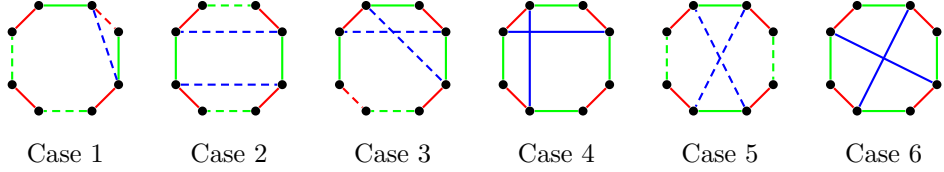
\begin{figure}[h]
	\centering
	\begin{tikzpicture}[scale=0.7, every node/.style={circle, fill, inner sep=1.2pt}]
	
	\node[fill=black] (v1)at(0.5,1) {};
	\node[fill=black] (v2) at (1, 0.5) {};
	\node[fill=black] (v3) at (1, -0.5) {};
	\node[fill=black] (v4) at (0.5, -1) {};
	\node[fill=black] (v5) at (-0.5, -1) {};
	\node[fill=black] (v6) at (-1, -0.5) {};
        \node[fill=black] (v7) at (-1, 0.5) {};
        \node[fill=black] (v8) at (-0.5, 1) {};
	\draw[densely dashed, red, thick] (v1) -- (v2);
	\draw[green, thick] (v2) -- (v3);
	\draw[red, thick] (v3) -- (v4);
	\draw[densely dashed, green, thick] (v4) -- (v5);
        \draw[red, thick] (v6) -- (v5);
        \draw[densely dashed, green, thick] (v6) -- (v7);
        \draw[red, thick] (v7) -- (v8);
        \draw[green, thick] (v8) -- (v1);
        \draw[densely dashed, blue, thick] (v1) -- (v3);
	\node[draw=none,fill=none] at (0, -1.8) {Case 1};
\end{tikzpicture}
\hspace{0.4cm}
	\begin{tikzpicture}[scale=0.7, every node/.style={circle, fill, inner sep=1.2pt}]
	
	\node[fill=black] (v1)at(0.5,1) {};
	\node[fill=black] (v2) at (1, 0.5) {};
	\node[fill=black] (v3) at (1, -0.5) {};
	\node[fill=black] (v4) at (0.5, -1) {};
	\node[fill=black] (v5) at (-0.5, -1) {};
	\node[fill=black] (v6) at (-1, -0.5) {};
        \node[fill=black] (v7) at (-1, 0.5) {};
        \node[fill=black] (v8) at (-0.5, 1) {};
	\draw[red, thick] (v1) -- (v2);
	\draw[green, thick] (v2) -- (v3);
	\draw[red, thick] (v3) -- (v4);
	\draw[densely dashed, green, thick] (v4) -- (v5);
        \draw[red, thick] (v6) -- (v5);
        \draw[green, thick] (v6) -- (v7);
        \draw[red, thick] (v7) -- (v8);
        \draw[densely dashed, green, thick] (v8) -- (v1);
        \draw[densely dashed, blue, thick] (v6) -- (v3);
        \draw[densely dashed, blue, thick] (v2) -- (v7);
	\node[draw=none,fill=none] at (0, -1.8) {Case 2};
\end{tikzpicture}
	\hspace{0.4cm}
	\begin{tikzpicture}[scale=0.7, every node/.style={circle, fill, inner sep=1.2pt}]
	
	\node[fill=black] (v1)at(0.5,1) {};
	\node[fill=black] (v2) at (1, 0.5) {};
	\node[fill=black] (v3) at (1, -0.5) {};
	\node[fill=black] (v4) at (0.5, -1) {};
	\node[fill=black] (v5) at (-0.5, -1) {};
	\node[fill=black] (v6) at (-1, -0.5) {};
        \node[fill=black] (v7) at (-1, 0.5) {};
        \node[fill=black] (v8) at (-0.5, 1) {};
	\draw[red, thick] (v1) -- (v2);
	\draw[green, thick] (v2) -- (v3);
	\draw[red, thick] (v3) -- (v4);
	\draw[densely dashed, green, thick] (v4) -- (v5);
        \draw[densely dashed, red, thick] (v6) -- (v5);
        \draw[green, thick] (v6) -- (v7);
        \draw[red, thick] (v7) -- (v8);
        \draw[green, thick] (v8) -- (v1);
        \draw[densely dashed, blue, thick] (v8) -- (v3);
        \draw[densely dashed, blue, thick] (v2) -- (v7);
	\node[draw=none,fill=none] at (0, -1.8) {Case 3};
\end{tikzpicture}
\hspace{0.4cm}
	\begin{tikzpicture}[scale=0.7, every node/.style={circle, fill, inner sep=1.2pt}]
	
	\node[fill=black] (v1)at(0.5,1) {};
	\node[fill=black] (v2) at (1, 0.5) {};
	\node[fill=black] (v3) at (1, -0.5) {};
	\node[fill=black] (v4) at (0.5, -1) {};
	\node[fill=black] (v5) at (-0.5, -1) {};
	\node[fill=black] (v6) at (-1, -0.5) {};
        \node[fill=black] (v7) at (-1, 0.5) {};
        \node[fill=black] (v8) at (-0.5, 1) {};
	\draw[red, thick] (v1) -- (v2);
	\draw[green, thick] (v2) -- (v3);
	\draw[red, thick] (v3) -- (v4);
	\draw[green, thick] (v4) -- (v5);
        \draw[red, thick] (v6) -- (v5);
        \draw[green, thick] (v6) -- (v7);
        \draw[red, thick] (v7) -- (v8);
        \draw[green, thick] (v8) -- (v1);
        \draw[blue, thick] (v2) -- (v7);
        \draw[blue, thick] (v8) -- (v5);
	\node[draw=none,fill=none] at (0, -1.8) {Case 4};
\end{tikzpicture}
\hspace{0.4cm}
	\begin{tikzpicture}[scale=0.7, every node/.style={circle, fill, inner sep=1.2pt}]
	
	\node[fill=black] (v1)at(0.5,1) {};
	\node[fill=black] (v2) at (1, 0.5) {};
	\node[fill=black] (v3) at (1, -0.5) {};
	\node[fill=black] (v4) at (0.5, -1) {};
	\node[fill=black] (v5) at (-0.5, -1) {};
	\node[fill=black] (v6) at (-1, -0.5) {};
        \node[fill=black] (v7) at (-1, 0.5) {};
        \node[fill=black] (v8) at (-0.5, 1) {};
	\draw[red, thick] (v1) -- (v2);
	\draw[densely dashed, green, thick] (v2) -- (v3);
	\draw[red, thick] (v3) -- (v4);
	\draw[green, thick] (v4) -- (v5);
        \draw[red, thick] (v6) -- (v5);
        \draw[densely dashed, green, thick] (v6) -- (v7);
        \draw[red, thick] (v7) -- (v8);
        \draw[green, thick] (v8) -- (v1);
        \draw[densely dashed, blue, thick] (v1) -- (v5);
        \draw[densely dashed, blue, thick] (v4) -- (v8);
	\node[draw=none,fill=none] at (0, -1.8) {Case 5};
\end{tikzpicture}
\hspace{0.4cm}
	\begin{tikzpicture}[scale=0.7, every node/.style={circle, fill, inner sep=1.2pt}]
	
	\node[fill=black] (v1)at(0.5,1) {};
	\node[fill=black] (v2) at (1, 0.5) {};
	\node[fill=black] (v3) at (1, -0.5) {};
	\node[fill=black] (v4) at (0.5, -1) {};
	\node[fill=black] (v5) at (-0.5, -1) {};
	\node[fill=black] (v6) at (-1, -0.5) {};
        \node[fill=black] (v7) at (-1, 0.5) {};
        \node[fill=black] (v8) at (-0.5, 1) {};
	\draw[red, thick] (v1) -- (v2);
	\draw[green, thick] (v2) -- (v3);
	\draw[red, thick] (v3) -- (v4);
	\draw[green, thick] (v4) -- (v5);
        \draw[red, thick] (v6) -- (v5);
        \draw[green, thick] (v6) -- (v7);
        \draw[red, thick] (v7) -- (v8);
        \draw[green, thick] (v8) -- (v1);
        \draw[blue, thick] (v1) -- (v5);
        \draw[blue, thick] (v3) -- (v7);
	\node[draw=none,fill=none] at (0, -1.8) {Case 6};
\end{tikzpicture}
\vspace{-1em}
\caption{Cases of adding blue chords in $H_1$}
\label{badcase'}
\end{figure}

Now we can assume $H_1$ is a $C_8$ and blue edges are chords in the cycle. We define the length of chord as the length of shortest path between the endpoints of the chord along the cycle and do casework. First, there is no blue chord of length two, since otherwise we can find a pair of disjoint $P_2$ in $G$ corresponding to a blue-red $P_2$ and a pair of disjoint green edges (See the dashed edges in Case 1). 

If there are two blue chords of length three, by symmetry there are only three cases (See Case 2,3 and 4). However, in Case 2 and 3 we can find a pair of disjoint $P_2$ (See the dashed edges in Case 2 and 3), a contradiction. In Case 4, adding any third blue chord of length at least three will lead to Case 2 or Case 3. Thus, there are at most one chord of length 2.

So there are at least two blue chords of length four, by symmetry there are only two cases (See Case 5 and 6). In Case 5 there is a pair of disjoint $P_2$ (See the dashed edges in Case 5). Finally, in Case 6 adding a third blue chord of length at least three would lead to Case 5, a contradiction. 
\end{proof}

\vspace{0.6cm}

%\textcolor{red}{first $P_2,P_2$, and prove thm2 in the end, clarify our way to prove in the introduction}\\
\noindent \emphd{\textit{Proof of~\Cref{main}:}}

If there is no vertex of degree larger than $4$ in $G$, then the number of edges is no more than $4n/3$, contradicting~\Cref{contradiction}. So we can assume the maximum degree of $G$ is $k\geq 5$ and then let $d_G(v)=k$. Let $\{v,v_{2i-1},v_{2i}\mid i\in[k]\}$ be the edges incident to $v$.

When $k\geq 6$, we first claim that every other vertex in $G$ is a neighbor of $v$. Assume there is a non-neighbor $u$, by~\Cref{contradiction}, $d_G(u)\geq 2$. Take any two edges $e_1,e_2$ incident to $u$. There are at least $k-4\geq 2$ edges incident to $v$ disjoint with $e_1,e_2$ by linearity, which is a $P_2$ disjoint from $e_1$ and $e_2$, contradicting~\Cref{S2S2}. Hence $|E (G)|=\big(\sum_{u\in V(G)}d_G(u)\big)/3\leq ((n-1)/2+3(n-1))/3=7(n-1)/6<15n/11,$ contradicting~\Cref{contradiction}.

When $k= 5$, any non-neighbor $u$ of $v$ has degree no more than 2 by~\Cref{lm8}. Meanwhile, by~\Cref{degreelist}, $d_G(v_{2i-1})+d_G(v_{2i})\leq 8,i\in[5]$. So $|E (G)|\leq (5+2(n-11)+8\times 5)/3 = (2n+23)/3\leq 15n/11$. Note that equality holds only when $n=11$ and $d_G(v_{2i-1})+d_G(v_{2i})= 8, i \in[5]$, which implies that the minimum degree of $G$ is at least 3. It also implies that degree lists of edges incident to any degree 5 vertex are either $(5,5,3)$ or $(5,4,4)$.
\vspace{0.27cm}

Now we assume $|V(G)|=11,|E (G)|=15$ and show the uniqueness of the extremal graph. Let $V(G) = \{v_i\}_{0\le i \le 10}$. Assume $d_G(v_0)=5$ and $v_0v_9v_{10}$ is an edge. If there exists edge of degree list $(5,5,3)$, we let $d_G(v_0v_9v_{10}) = (5,5,3)$, and $(5,4,4)$ otherwise. 

By deleting three vertices $v_0,v_{9},v_{10}$ from $G$, we obtain a 3-graph $G_1$ with 8 vertices and 4 edges. Up to an automorphism, there are 6 cases for $G_1$ as drawn in~\Cref{G1}. Moreover, we add the labeled 2-edges of link graph $L(G,\{v_0,v_9v_{10}\})$ into $G_1$. Thus, we get a non-uniform hypergraph $G_2$, where 2-edges are labeled in $G_2$. From now on, we always use red, green and blue to color edges labeled $v_0$, $v_9$ and $v_{10}$ respectively. By linearity, in $G_2$ monochromatic 2-edges form a matching, and further a perfect matching if the 2-edges are labeled by a degree 5 vertex. In particular, the red 2-edges form a perfect matching, and each edge connects vertices of degree 5 and 3 or 4 and 4. It implies the number of degree 5 vertices equals to degree 3 vertices. Also note that the minimal degree of $G_2$ is at least 3 since the minimum degree in $G$ is at least 3. 

\begin{figure}[h]
	\centering
	\begin{tikzpicture}[scale=0.7, every node/.style={circle, fill, inner sep=1.2pt}]
	\node[fill=black] (v1) at (-1, 0) {};
		\node[fill=black] (v2) at (0, 0) {};
		\node[fill=black] (v3) at (1, 0) {};
		\node[fill=black] (v4) at (-1, 1) {};
		\node[fill=black] (v5) at (-1, 2) {};
		\node[fill=black] (v6) at (-0.33, 0.67) {};
        \node[fill=black] (v7) at (1, 1) {};
		\node[fill=red] (v8) at (0, 2) {};

		\draw[black, thick] (v1) -- (v3);
		\draw[black, thick] (v1) -- (v5);
		\draw[black, thick] (v3) -- (v4);
		\draw[black, thick] (v2) -- (v5);
        
        \draw[gray] (v8) -- (v1);
        \draw[gray] (v8) -- (v2);
        \draw[gray] (v8) -- (v3);
        \draw[gray] (v8) -- (v4);
        \draw[gray] (v8) -- (v5);
        \draw[gray] (v8) -- (v6);
        \draw[gray] (v8) -- (v7);

	\node[draw=none,fill=none] at (0, -0.8) {Case 1};
\end{tikzpicture}
	\hspace{0.4cm}
	\begin{tikzpicture}[scale=0.7, every node/.style={circle, fill, inner sep=1.2pt}]
	
    \node[fill=black] (v1) at (-1, 0) {};
		\node[fill=black] (v2) at (0, 0) {};
		\node[fill=black] (v3) at (1, 0) {};
		\node[fill=black] (v4) at (-1, 1) {};
		\node[fill=black] (v5) at (-1, 2) {};
		\node[fill=black] (v6) at (0, 1) {};
		\node[fill=black] (v9) at (1, 2) {};
        \node[fill=red] (v7) at (0.4, 2) {};
		
		\draw[black, thick] (v1) -- (v3);
		\draw[black, thick] (v3) -- (v5);
		\draw[black, thick] (v1) -- (v9);
		\draw[black, thick] (v1) -- (v5);
        
        \draw[gray] (v7) -- (v2);
        \draw[gray] (v7) -- (v9);
        \draw[gray] (v7) -- (v3);
        \draw[gray] (v7) -- (v4);
        \draw[gray] (v7) -- (v5);
        \draw[gray] (v7) -- (v6);

	\node[draw=none,fill=none] at (0, -0.8) {Case 2};
\end{tikzpicture}
\hspace{0.4cm}
	\begin{tikzpicture}[scale=0.7, every node/.style={circle, fill, inner sep=1.2pt}]
	
	\node[fill=black] (v1) at (-1, 0) {};
		\node[fill=black] (v2) at (0, 1) {};
		\node[fill=black] (v3) at (1, 0) {};
		\node[fill=black] (v4) at (-1, 1) {};
		\node[fill=black] (v5) at (-1, 2) {};
		\node[fill=black] (v6) at (1, 1) {};
		\node[fill=black] (v7) at (1, 2) {};	
        \node[fill=red] (v8) at (0, 2) {};	
		
		\draw[black, thick] (v1) -- (v7);
		\draw[black, thick] (v1) -- (v5);
		\draw[black, thick] (v3) -- (v5);
		\draw[black, thick] (v3) -- (v7);

        \draw[gray] (v8) -- (v1);
        \draw[gray] (v8) -- (v7);
        \draw[gray] (v8) -- (v3);
        \draw[gray] (v8) -- (v4);
        \draw[gray] (v8) -- (v5);
        \draw[gray] (v8) -- (v6);
        
	\node[draw=none,fill=none] at (0, -0.8) {Case 3};
\end{tikzpicture}
\hspace{0.4cm}
	\begin{tikzpicture}[scale=0.7, every node/.style={circle, fill, inner sep=1.2pt}]
	
	\node[fill=black] (v1) at (-1, 0) {};
		\node[fill=black] (v2) at (0, 0) {};
		\node[fill=black] (v3) at (1, 0) {};
	
		\node[fill=black] (v5) at (-1, 2) {};
		\node[fill=black] (v6) at (0.1, 0.9) {};
		\node[fill=black] (v8) at (1, 1.15) {};
		\node[fill=black] (v9) at (1, 2.1) {};
		
		\draw[black, thick] (v1) -- (v3);
		\draw[black, thick] (v1) -- (v5);
		\draw[black, thick] (v3) -- (v5);
		\draw[black, thick] (v3) -- (v9);

    	\node[fill=red] (v4) at (-1, 1.05) {};
        \draw[gray] (v4) -- (v6);
        \draw[gray] (v4) -- (v2);
        \draw[gray] (v4) -- (v8);
        \draw[gray] (v4) -- (v9);

	\node[draw=none,fill=none] at (0, -0.8) {Case 4};
\end{tikzpicture}
\hspace{0.4cm}
	\begin{tikzpicture}[scale=0.7, every node/.style={circle, fill, inner sep=1.2pt}]
	
````````\node[fill=blue] (v1) at (-1, 0) {};
		
		\node[fill=blue] (v3) at (1, 0) {};
		\node[fill=black] (v5) at (-1, 2) {};
		\node[fill=black] (v6) at (1, 1.3) {};
		\node[fill=black] (v7) at (1, 2.1) {};
		
		\draw[black, thick] (v1) -- (v3);
		\draw[black, thick] (v1) -- (v5);
		\draw[black, thick] (v3) -- (v7);
        \node[fill=blue] (v2) at (0.2, 0) {};
		\draw[black, thick] (v5) -- (v2);

        \node[fill=red] (v4) at (-1, 0.9) {};
        \node[fill= red ] (v8) at (-0.3, 0.85) {};
        \draw[red] (v4) -- (v2);
		\draw[green] (v4) -- (v3);
        \draw[red] (v8) -- (v3);
		\draw[green] (v8) -- (v1);

        \draw[gray] (v4) -- (v6);
        \draw[gray] (v4) -- (v7);
        \draw[gray] (v4) -- (v8);
        \draw[gray] (v8) -- (v6);
        \draw[gray] (v8) -- (v7);

	\node[draw=none,fill=none] at (0, -0.8) {Case 5};
\end{tikzpicture}
\hspace{0.4cm}
	\begin{tikzpicture}[scale=0.7, every node/.style={circle, fill, inner sep=1.2pt}]
	
        \node[fill=black] (v9) at (1, 2) {};
		\node[fill=black] (v1) at (-1, 0) {};
		\node[fill=black] (v3) at (1, 0) {};
		\node[fill=black] (v5) at (-1, 2) {};
		\node[fill=black] (v7) at (1, 2) {};

		\draw[black, thick] (v1) -- (v3);
		\draw[black, thick] (v9) -- (v5);
		\draw[black, thick] (v3) -- (v9);
		\draw[black, thick] (v1) -- (v5);
        
        \node[fill=red] (v2) at (0, 0) {};
		\node[fill=red] (v4) at (-1, 1) {};     
        \node[fill=red] (v6) at (1, 1) {};
        \node[fill=red] (v8) at (0, 2) {};
        
        \draw[gray] (v2) -- (v6);
        \draw[gray] (v6) -- (v8);
        \draw[gray] (v4) -- (v8);
        \draw[gray] (v4) -- (v2);

        \draw[densely dashed, black, thick] (v8) -- (v3);
		\draw[densely dashed, black, thick] (v8) -- (v2);
		\draw[densely dashed, black, thick] (v8) -- (v1);
        
	\node[draw=none,fill=none] at (0, -0.8) {Case 6};
\end{tikzpicture}
\vspace{-1em}
\caption{All cases of $G_1$ and the gray pairs are disjoint from some $P_2$}
\label{G1}
\end{figure}
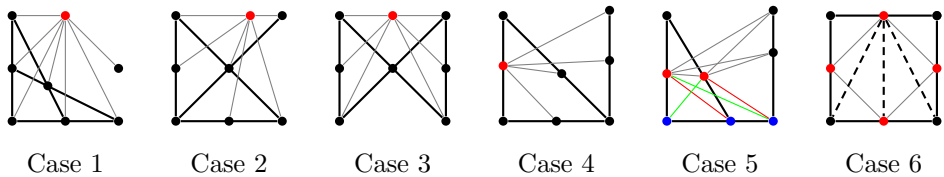

The key observation is that a $P_2$ in $G_1$ cannot be disjoint from a colored edge in $G_2$. Otherwise the colored edge together with $\{v_0,v_9,v_{10}\}$ form a second $P_2$ in $G$ disjoint from the first one, contradicting \Cref{2P2}. As a result, the gray pairs in~\Cref{G1} cannot be used as a colored edge, which implies the red vertex in Case 1,2,3,4 has degree smaller than 3 in $G_2$, contradicting the minimum degree of $G_2$. 

In Case 5, the two red vertices (see~\Cref{G1}, Case 5) have degree at most 3 so the minimum degree in $G_2$, and thus $G$, is exactly 3. So we know that $d_G(v_0v_9v_{10}) = (5,5,3)$. As a result, red and green edges form two perfect matchings in $G_2$. So, the colored edges incident to red vertices are red and green (see~\Cref{G1}, Case 5).

Meanwhile, since $v_0$ and $v_9$ are degree 5 vertices in $G$, the red and green edges must connect vertices of degree 3 and 5 or of degree 4 and 4 in $G_2$. Now that blue vertices are connected to red vertices of degree 3. It implies that blue vertices all have degree 5 while they are in the same edge, contradicting~\Cref{degreelist}.

Now the only possible case is Case 6, we can see that the red vertices in Case 6 (See~\Cref{G1}) have degree at most 4 for the same reason. Moreover, if a red vertex reaches degree 4, then its three incident colored edges is unique (See dashed edges drawn in Case 6). Next we will need more careful argument to determine the structure of $G_2$.
 \begin{figure}[h]
    \centering
\begin{tikzpicture}[scale=0.7, every node/.style={circle, fill, inner sep=1.2pt}]
        \node[fill=black] (v9) at (1, 2) {};
		\node[fill=black] (v1) at (-1, 0) {};
		\node[fill=black] (v3) at (1, 0) {};
		\node[fill=black] (v5) at (-1, 2) {};
		\node[fill=black] (v7) at (1, 2) {};

		\draw[black, thick] (v1) -- (v3);
		\draw[black, thick] (v9) -- (v5);
		\draw[black, thick] (v3) -- (v9);
		\draw[black, thick] (v1) -- (v5);

        \node[fill=red] (v2) at (0, 0) {};
		\node[fill=red] (v4) at (-1, 1) {};     
        \node[fill=red] (v6) at (1, 1) {};
        \node[fill=red] (v8) at (0, 2) {};

        \draw[red, thick] (v8) -- (v3);
        \draw[red, thick] (v5) -- (v2);
        \draw[red, thick] (v1) -- (v6);
        \draw[red, thick] (v7) -- (v4);
\end{tikzpicture}
\hspace{0.6cm}
\begin{tikzpicture}[scale=0.7, every node/.style={circle, fill, inner sep=1.2pt}]
    
        \node[fill=black] (v9) at (1, 2) {};
		\node[fill=black] (v1) at (-1, 0) {};
		\node[fill=black] (v3) at (1, 0) {};
		\node[fill=black] (v5) at (-1, 2) {};
		\node[fill=black] (v7) at (1, 2) {};

		\draw[black, thick] (v1) -- (v3);
		\draw[black, thick] (v9) -- (v5);
		\draw[black, thick] (v3) -- (v9);
		\draw[black, thick] (v1) -- (v5);

        \node[fill=red] (v2) at (0, 0) {};
		\node[fill=red] (v4) at (-1, 1) {};     
        \node[fill=red] (v6) at (1, 1) {};
        \node[fill=red] (v8) at (0, 2) {}; 

		\draw[red, thick] (v5) -- (v2);
        \draw[red, thick] (v8) -- (v3);
        \draw[red, thick] (v4) -- (v6);
        \draw[red, thick] (v7) -- (v1);

\end{tikzpicture}
\hspace{0.6cm}
\begin{tikzpicture}[scale=0.7, every node/.style={circle, fill, inner sep=1.2pt}]
    
        \node[fill=black] (v9) at (1, 2) {};
		\node[fill=black] (v1) at (-1, 0) {};
		\node[fill=black] (v3) at (1, 0) {};
		\node[fill=black] (v5) at (-1, 2) {};
		\node[fill=black] (v7) at (1, 2) {};

		\draw[black, thick] (v1) -- (v3);
		\draw[black, thick] (v9) -- (v5);
		\draw[black, thick] (v3) -- (v9);
		\draw[black, thick] (v1) -- (v5);

        \node[fill=red] (v2) at (0, 0) {};
		\node[fill=red] (v4) at (-1, 1) {};     
        \node[fill=red] (v6) at (1, 1) {};
        \node[fill=red] (v8) at (0, 2) {}; 

		\draw[red, thick] (v8) -- (v2);
        \draw[red, thick] (v5) -- (v3);
        \draw[red, thick] (v4) -- (v6);
        \draw[red, thick] (v7) -- (v1);

\end{tikzpicture}

\caption{Perfect matchings in Case 6}
\label{perfectmatch}
\end{figure}
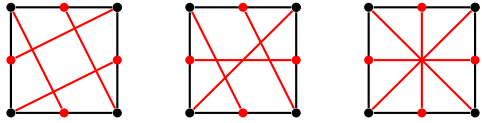

To start with, since $v_0$ is a degree 5 vertex, there are are 4 red 2-edges in $G_2$ which form a perfect matching. However in Case 6, up to an homomorphism, there are only three kinds of perfect matching as drawn in~\Cref{perfectmatch} (Note that grey pairs in Case 6 in~\Cref{G1} cannot be colored).
\begin{figure}[h]
    \centering
\begin{tikzpicture}[scale=0.7, every node/.style={circle, fill, inner sep=1.2pt}]
	\node[draw=none,fill=none] at (-1.4, 2.4) {$v_1$};
    \node[draw=none,fill=none] at (0, 2.4) {$v_2$};
    \node[draw=none,fill=none] at (1.4, 2.4) {$v_3$};
    \node[draw=none,fill=none] at (-1.5, 1) {$v_4$};
    \node[draw=none,fill=none] at (1.5, 1) {$v_5$};
    \node[draw=none,fill=none] at (-1.5, -0.4) {$v_6$};
    \node[draw=none,fill=none] at (0, -0.4) {$v_7$};
    \node[draw=none,fill=none] at (1.5, -0.4) {$v_8$};
    
        \node[fill=black] (v9) at (1, 2) {};
		\node[fill=black] (v1) at (-1, 0) {};
		\node[fill=black] (v3) at (1, 0) {};
		\node[fill=black] (v5) at (-1, 2) {};
		\node[fill=black] (v7) at (1, 2) {};

		\draw[black, thick] (v1) -- (v3);
		\draw[black, thick] (v9) -- (v5);
		\draw[black, thick] (v3) -- (v9);
		\draw[black, thick] (v1) -- (v5);

        \node[fill=red] (v2) at (0, 0) {};
		\node[fill=red] (v4) at (-1, 1) {};     
        \node[fill=red] (v6) at (1, 1) {};
        \node[fill=red] (v8) at (0, 2) {}; 

        \draw[densely dashed, black, thick] (v8) -- (v3);
		\draw[densely dashed,black, thick] (v8) -- (v2);
		\draw[densely dashed,black, thick] (v8) -- (v1);
        
        \draw[densely dashed,black, thick] (v2) -- (v5);
		\draw[densely dashed,black, thick] (v9) -- (v2);

        \draw[densely dashed,black, thick] (v1) -- (v6);
		\draw[densely dashed,black, thick] (v5) -- (v6);
		\draw[densely dashed,black, thick] (v4) -- (v6);

        \draw[densely dashed,black, thick] (v4) -- (v3);
		\draw[densely dashed,black, thick] (v4) -- (v7);

    \node[draw=none,fill=none] at (3.6, 2.4) {$v_1$};
    \node[draw=none,fill=none] at (5, 2.4) {$v_2$};
    \node[draw=none,fill=none] at (6.4, 2.4) {$v_3$};
    \node[draw=none,fill=none] at (3.5, 1) {$v_4$};
    \node[draw=none,fill=none] at (6.5, 1) {$v_5$};
    \node[draw=none,fill=none] at (3.5, -0.4) {$v_6$};
    \node[draw=none,fill=none] at (5, -0.4) {$v_7$};
    \node[draw=none,fill=none] at (6.5, -0.4) {$v_8$};
    
        \node[fill=black] (u9) at (6, 2) {};
		\node[fill=black] (u1) at (4, 0) {};
		\node[fill=black] (u3) at (6, 0) {};
		\node[fill=black] (u5) at (4, 2) {};
		\node[fill=black] (u7) at (6, 2) {};

		\draw[black, thick] (u1) -- (u3);
		\draw[black, thick] (u9) -- (u5);
		\draw[black, thick] (u3) -- (u9);
		\draw[black, thick] (u1) -- (u5);

        \node[fill=red] (u2) at (5, 0) {};
		\node[fill=red] (u4) at (4, 1) {};     
        \node[fill=red] (u6) at (6, 1) {};
        \node[fill=red] (u8) at (5, 2) {}; 

        \draw[densely dashed, black, thick] (u2) -- (u8);
		\draw[red, thick] (u5) -- (u2);
		\draw[densely dashed, black, thick] (u5) -- (u6);
        \draw[densely dashed, black, thick] (u3) -- (u4);
		\draw[red, thick] (u3) -- (u8);
		\draw[densely dashed, black, thick] (u1) -- (u8);
        \draw[densely dashed, black, thick] (u4) -- (u6);
        \draw[red, thick] (u1) -- (u6);
        \draw[red, thick] (u4) -- (u9);
        \draw[densely dashed, black, thick] (u2) -- (u9);
\end{tikzpicture}
\caption{Cases of all degree 4 vertices}
\label{Case6-12}
\end{figure}
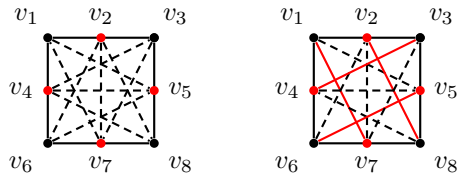

We first suppose every vertex in $G_2$ has degree 4 and thus the colored 2-edges incident to $v_2,v_4,v_5,v_7$ are uniquely determined (there are at most 3 possible 2-edges incident to each of these vertices). Then all the ten colored edges are determined (See dashed edges in the first graph of~\Cref{Case6-12}). Note that there are no diagonal $v_3v_6$ or $v_1v_8$ in these ten edges, and thus only the first graph in Figure 9 satisfied this condition. So we may assume $v_1v_7, v_2v_8, v_3v_4, v_5v_6$ are red. Note that after removing these four red edges, the green-blue edges in the first graph of ~\Cref{Case6-12} will form two disjoint $P_4's$: $v_1v_5v_4v_8$ and $v_6v_2v_7v_3$. Then we can find a green-blue $P_2$ of $v_1v_5v_4$ disjoint from a black-red $P_2$ of $v_6v_7v_8v_2$, contradicting~\Cref{2P2}.

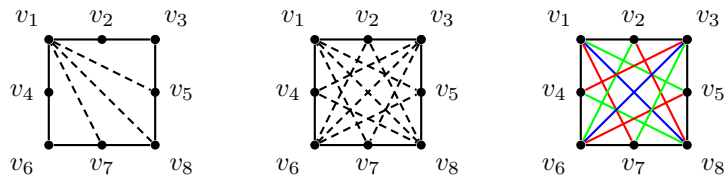
\begin{figure}[h!]
    \centering
\begin{tikzpicture}[scale=0.7, every node/.style={circle, fill, inner sep=1.2pt}]
\node[draw=none,fill=none] at (-6.4, 2.4) {$v_1$};
    \node[draw=none,fill=none] at (-5, 2.4) {$v_2$};
    \node[draw=none,fill=none] at (-3.6, 2.4) {$v_3$};
    \node[draw=none,fill=none] at (-6.5, 1) {$v_4$};
    \node[draw=none,fill=none] at (-3.5, 1) {$v_5$};
    \node[draw=none,fill=none] at (-6.5, -0.4) {$v_6$};
    \node[draw=none,fill=none] at (-5, -0.4) {$v_7$};
    \node[draw=none,fill=none] at (-3.5, -0.4) {$v_8$};
    
        \node[fill=black] (v9) at (-4, 2) {};
		\node[fill=black] (v1) at (-6, 0) {};
		\node[fill=black] (v3) at (-4, 0) {};
		\node[fill=black] (v5) at (-6, 2) {};
		\node[fill=black] (v7) at (-4, 2) {};

		\draw[black, thick] (v1) -- (v3);
		\draw[black, thick] (v9) -- (v5);
		\draw[black, thick] (v3) -- (v9);
		\draw[black, thick] (v1) -- (v5);

        \node[fill=black] (v2) at (-5, 0) {};
		\node[fill=black] (v4) at (-6, 1) {};     
        \node[fill=black] (v6) at (-4, 1) {};
        \node[fill=black] (v8) at (-5, 2) {}; 

        \draw[densely dashed, black, thick] (v5) -- (v3);
        \draw[densely dashed,black, thick] (v2) -- (v5);
		\draw[densely dashed,black, thick] (v5) -- (v6);

	\node[draw=none,fill=none] at (-1.4, 2.4) {$v_1$};
    \node[draw=none,fill=none] at (0, 2.4) {$v_2$};
    \node[draw=none,fill=none] at (1.4, 2.4) {$v_3$};
    \node[draw=none,fill=none] at (-1.5, 1) {$v_4$};
    \node[draw=none,fill=none] at (1.5, 1) {$v_5$};
    \node[draw=none,fill=none] at (-1.5, -0.4) {$v_6$};
    \node[draw=none,fill=none] at (0, -0.4) {$v_7$};
    \node[draw=none,fill=none] at (1.5, -0.4) {$v_8$};
    
        \node[fill=black] (v9) at (1, 2) {};
		\node[fill=black] (v1) at (-1, 0) {};
		\node[fill=black] (v3) at (1, 0) {};
		\node[fill=black] (v5) at (-1, 2) {};
		\node[fill=black] (v7) at (1, 2) {};

		\draw[black, thick] (v1) -- (v3);
		\draw[black, thick] (v9) -- (v5);
		\draw[black, thick] (v3) -- (v9);
		\draw[black, thick] (v1) -- (v5);

        \node[fill=black] (v2) at (0, 0) {};
		\node[fill=black] (v4) at (-1, 1) {};     
        \node[fill=black] (v6) at (1, 1) {};
        \node[fill=black] (v8) at (0, 2) {}; 

        \draw[densely dashed, black, thick] (v8) -- (v3);
		\draw[densely dashed,black, thick] (v5) -- (v3);
		\draw[densely dashed,black, thick] (v8) -- (v1);
        
        \draw[densely dashed,black, thick] (v2) -- (v5);
		\draw[densely dashed,black, thick] (v9) -- (v2);

        \draw[densely dashed,black, thick] (v1) -- (v6);
		\draw[densely dashed,black, thick] (v5) -- (v6);
		\draw[densely dashed,black, thick] (v7) -- (v1);

        \draw[densely dashed,black, thick] (v4) -- (v3);
		\draw[densely dashed,black, thick] (v4) -- (v7);
    
    \node[draw=none,fill=none] at (3.6, 2.4) {$v_1$};
    \node[draw=none,fill=none] at (5, 2.4) {$v_2$};
    \node[draw=none,fill=none] at (6.4, 2.4) {$v_3$};
    \node[draw=none,fill=none] at (3.5, 1) {$v_4$};
    \node[draw=none,fill=none] at (6.5, 1) {$v_5$};
    \node[draw=none,fill=none] at (3.5, -0.4) {$v_6$};
    \node[draw=none,fill=none] at (5, -0.4) {$v_7$};
    \node[draw=none,fill=none] at (6.5, -0.4) {$v_8$};
    
        \node[fill=black] (u9) at (6, 2) {};
		\node[fill=black] (u1) at (4, 0) {};
		\node[fill=black] (u3) at (6, 0) {};
		\node[fill=black] (u5) at (4, 2) {};
		\node[fill=black] (u7) at (6, 2) {};

		\draw[black, thick] (u1) -- (u3);
		\draw[black, thick] (u9) -- (u5);
		\draw[black, thick] (u3) -- (u9);
		\draw[black, thick] (u1) -- (u5);

        \node[fill=black] (u2) at (5, 0) {};
		\node[fill=black] (u4) at (4, 1) {};     
        \node[fill=black] (u6) at (6, 1) {};
        \node[fill=black] (u8) at (5, 2) {}; 

        \draw[blue, thick] (u5) -- (u3);
		\draw[red, thick] (u5) -- (u2);
		\draw[green, thick] (u5) -- (u6);
        \draw[green, thick] (u3) -- (u4);
		\draw[red, thick] (u3) -- (u8);
		\draw[green, thick] (u1) -- (u8);
        \draw[blue, thick] (u1) -- (u9);
        \draw[red, thick] (u1) -- (u6);
        \draw[red, thick] (u4) -- (u9);
        \draw[green, thick] (u2) -- (u9);
\end{tikzpicture}
\caption{Case 6-3}
\label{Case6-3}
\end{figure}

Since the red edges form a matching with end points of degree $(5,3)$ or $(4,4)$ and by the discussion above not all vertices in $G_2$ have degree 4, there exists 
a degree 5 vertex in $G_2$. By our choices of $v_0v_9v_{10}$, it follows that $v_0v_9v_{10}$ is $(5,5,3)$.
Note that edges colored red and green both form a perfect matching, and for any degree 5 vertex in $G_2$, the end points in these two perfect matchings have degree 3, and vice versa. That is, the number of degree 5 vertices equals to degree 3 vertices. Thus there are at least 2 vertices in $G_2$ of degree 3 and thus at least 2 vertices in $G_2$ of degree 5.

Recall that $v_2,v_5,v_4,v_7$ have degree at most 4, so the degree 5 vertices can only be $v_1,v_3,v_6$ or $v_8$. We will show that not even two but all of these four have to be degree 5. Without loss of generality, assume $d_{G_2}(v_1)=5$. Note that $v_1$ has three non-neighbor in $G_1$, so the colored edges incident to $v_1$ is uniquely determined (See the dashed edges in the first figure in~\Cref{Case6-3}). Since we assume the degree of $v_1$ to be 5, two of its neighbors (connected by red and green edges) should be degree 3. So two of $v_5,v_8,v_7$ are degree 3. Then by pigeonhole principle at least one of $v_5,v_8$ and one of $v_7,v_8$ is degree 3. So $d(v_3v_5v_8)$ and $d(v_6v_7v_8)$ cannot be $(5,4,4)$. 

Assume $d_{G_2}(v_3)=5$, then $d(v_3v_5v_8)=(5,5,3)$. While $d_{G_2}(v_5)\le4$, so $v_5$ has to be the degree 3 vertex and further $d_{G_2}(v_3)=d_{G_2}(v_8)=5$. So $d_{G_2}(v_3)=5$ implies $d_{G_2}(v_8)=5$, and vice versa. By symmetry it also works for $v_6$ and $v_8$. Recall that there is a second vertex of degree 5 in $v_1,v_3,v_6$ or $v_8$, we conclude that all these four vertex have degree 5, and thus the ten colored edges are determined (see dashed edges in the second graph of~\Cref{Case6-3}).

Since $v_2v_7$ and $v_4v_5$ are not colored, among all the perfect matching as shown in~\Cref{perfectmatch}, only the first type is possible to be colored red. By symmetry we can assume $v_1v_7,v_2v_8,v_3v_4,v_5v_7$ are colored red. Then similarly the perfect matching colored green can only chosen to be $v_1v_3,v_2v_6,v_3v_7,v_4v_8$ and the two edges colored blue must be $v_1v_8,v_3v_6$ as drawn in the second figure in~\Cref{Case6-3}. We can check that this is the unique extremal graph $G_0$ we defined in section 1.\hfill\qedsymbol\\

In a summary, in this paper we show that the unique extremal linear 3-graphs that does not contain $P_5$ is the disjoint union of an 11 vertex graph $G_0$. But the general lower bound is obtained from disjoint union of maximal partial steiner system on $2k$ vertices. Maybe there are more general extremal graph on $2k+1$ vertices that can beat the bound for $MPTS(2k)$ for bigger $k$. We will keep it as an open question.

\nocite{*}
\bibliographystyle{plain}
\bibliography{ref}

@inbook{ea299c3877e64a5b996676debaabb8d3,
title = "Handbook of combinatorial designs",
author = "Charles Colbourn and Dinitz, {Jeffrey H.}",
note = "Publisher Copyright: {\textcopyright} 2007 by Taylor \& Francis Group, LLC.",
year = "2006",
month = jan,
day = "1",
language = "English (US)",
isbn = "9781584885061",
pages = "1--984",
booktitle = "Handbook of Combinatorial Designs, Second Edition",
publisher = "CRC Press",
}

@article {MR4477845,
    AUTHOR = {Tang, Chaoliang and Wu, Hehui and Zhang, Shengtong and Zheng,
              Zeyu},
     TITLE = {On the {T}ur\'an number of the linear 3-graph {$C_{13}$}},
   JOURNAL = {Electron. J. Combin.},
  FJOURNAL = {Electronic Journal of Combinatorics},
    VOLUME = {29},
      YEAR = {2022},
    NUMBER = {3},
     PAGES = {Paper No. 3.46, 6},
      ISSN = {1077-8926},
   MRCLASS = {05C35 (05C65 05D05)},
  MRNUMBER = {4477845},
MRREVIEWER = {G\'abor\ N.\ S\'ark\"ozy},
       DOI = {10.37236/10775},
       URL = {https://doi.org/10.37236/10775},
}

@article {MR4315540,
    AUTHOR = {Gy\'arf\'as, Andr\'as and Ruszink\'o, Mikl\'os and
              S\'ark\"ozy, G\'abor N.},
     TITLE = {Linear {T}ur\'an numbers of acyclic triple systems},
   JOURNAL = {European J. Combin.},
  FJOURNAL = {European Journal of Combinatorics},
    VOLUME = {99},
      YEAR = {2022},
     PAGES = {Paper No. 103435, 12},
      ISSN = {0195-6698,1095-9971},
   MRCLASS = {05C35},
  MRNUMBER = {4315540},
       DOI = {10.1016/j.ejc.2021.103435},
       URL = {https://doi.org/10.1016/j.ejc.2021.103435},
}

@article {MR4192054,
    AUTHOR = {Gy\'arf\'as, Andr\'as and S\'ark\"ozy, G\'abor N.},
     TITLE = {Tur\'an and {R}amsey numbers in linear triple systems},
   JOURNAL = {Discrete Math.},
  FJOURNAL = {Discrete Mathematics},
    VOLUME = {344},
      YEAR = {2021},
    NUMBER = {3},
     PAGES = {Paper No. 112258, 12},
      ISSN = {0012-365X,1872-681X},
   MRCLASS = {05C65 (05B30 05C35 05C55 05D10)},
  MRNUMBER = {4192054},
MRREVIEWER = {Yuval\ Wigderson},
       DOI = {10.1016/j.disc.2020.112258},
       URL = {https://doi.org/10.1016/j.disc.2020.112258},
}

@article {MR2433776,
    AUTHOR = {Balister, P. N. and Gy\H{o}ri, E. and Lehel, J. and Schelp, R.
              H.},
     TITLE = {Connected graphs without long paths},
   JOURNAL = {Discrete Math.},
  FJOURNAL = {Discrete Mathematics},
    VOLUME = {308},
      YEAR = {2008},
    NUMBER = {19},
     PAGES = {4487--4494},
      ISSN = {0012-365X,1872-681X},
   MRCLASS = {05C55 (05C38)},
  MRNUMBER = {2433776},
MRREVIEWER = {Adam\ Pawe\l\ Wojda},
       DOI = {10.1016/j.disc.2007.08.047},
       URL = {https://doi.org/10.1016/j.disc.2007.08.047},
}

@article {MR2900051,
    AUTHOR = {Dellamonica, Jr., D. and Haxell, P. and \L uczak, T. and
              Mubayi, D. and Nagle, B. and Person, Y. and R\"odl, V. and
              Schacht, M. and Verstra\"ete, J.},
     TITLE = {On even-degree subgraphs of linear hypergraphs},
   JOURNAL = {Combin. Probab. Comput.},
  FJOURNAL = {Combinatorics, Probability and Computing},
    VOLUME = {21},
      YEAR = {2012},
    NUMBER = {1-2},
     PAGES = {113--127},
      ISSN = {0963-5483,1469-2163},
   MRCLASS = {05C65},
  MRNUMBER = {2900051},
MRREVIEWER = {S\`onia\ P.\ Mansilla},
       DOI = {10.1017/S0963548311000575},
       URL = {https://doi.org/10.1017/S0963548311000575},
}

@article {MR4473024,
    AUTHOR = {Gao, Guorong and Chang, An and Sun, Qi},
     TITLE = {Asymptotic {T}ur\'an number for linear 5-cycle in 3-uniform
              linear hypergraphs},
   JOURNAL = {Discrete Math.},
  FJOURNAL = {Discrete Mathematics},
    VOLUME = {346},
      YEAR = {2023},
    NUMBER = {1},
     PAGES = {Paper No. 113128, 8},
      ISSN = {0012-365X,1872-681X},
   MRCLASS = {05C35 (05C65)},
  MRNUMBER = {4473024},
       DOI = {10.1016/j.disc.2022.113128},
       URL = {https://doi.org/10.1016/j.disc.2022.113128},
}

@article {MR4498689,
    AUTHOR = {Carbonero, Alvaro and Fletcher, Willem and Guo, Jing and
              Gy\'arf\'as, Andr\'as and Wang, Rona and Yan, Shiyu},
     TITLE = {Crowns in linear 3-graphs of minimum degree 4},
   JOURNAL = {Electron. J. Combin.},
  FJOURNAL = {Electronic Journal of Combinatorics},
    VOLUME = {29},
      YEAR = {2022},
    NUMBER = {4},
     PAGES = {Paper No. 4.17, 8},
      ISSN = {1077-8926},
   MRCLASS = {05C35 (05B07 05D05)},
  MRNUMBER = {4498689},
MRREVIEWER = {W.\ G.\ Brown},
       DOI = {10.37236/11037},
       URL = {https://doi.org/10.37236/11037},
}

@misc{fletcher2021improvedupperboundlinear,
      title={Improved Upper Bound on the Linear Tur\'an Number of the Crown}, 
      author={Willem Fletcher},
      year={2021},
      eprint={2109.02729},
      archivePrefix={arXiv},
      primaryClass={math.CO},
      url={https://arxiv.org/abs/2109.02729}, 
}

@incollection {MR519318,
    AUTHOR = {Ruzsa, I. Z. and Szemer\'edi, E.},
     TITLE = {Triple systems with no six points carrying three triangles},
 BOOKTITLE = {Combinatorics ({P}roc. {F}ifth {H}ungarian {C}olloq.,
              {K}eszthely, 1976), {V}ol. {II}},
    SERIES = {Colloq. Math. Soc. J\'anos Bolyai},
    VOLUME = {18},
     PAGES = {939--945},
 PUBLISHER = {North-Holland, Amsterdam-New York},
      YEAR = {1978},
      ISBN = {0-444-85093-3},
   MRCLASS = {05C99 (05B30)},
  MRNUMBER = {519318},
MRREVIEWER = {W.\ G.\ Brown},
}

@article {MR4835949,
    AUTHOR = {Gy\H{o}ri, Ervin and Salia, Nika},
     TITLE = {Linear three-uniform hypergraphs with no {B}erge path of given
              length},
   JOURNAL = {J. Combin. Theory Ser. B},
  FJOURNAL = {Journal of Combinatorial Theory. Series B},
    VOLUME = {171},
      YEAR = {2025},
     PAGES = {36--48},
      ISSN = {0095-8956,1096-0902},
   MRCLASS = {05C35 (05C65)},
  MRNUMBER = {4835949},
       DOI = {10.1016/j.jctb.2024.11.003},
       URL = {https://doi.org/10.1016/j.jctb.2024.11.003},
}

\end{document}